\documentclass[12pt]{amsart}

\usepackage{amsmath, amscd, latexsym, hyperref, times, color}
\usepackage{subcaption, xcolor, graphicx, import, comment, calc}

\oddsidemargin 0in \theoremstyle{plain} \topmargin 0in

\oddsidemargin .25in \theoremstyle{plain}

\newtheorem{Thm}{Theorem}
\numberwithin{Thm}{section}
\newtheorem{Lem}[Thm]{Lemma}
\newtheorem{Conj}[Thm]{Conjecture}
\newtheorem{Clm}[Thm]{Claim}
\newtheorem{Cor}[Thm]{Corollary}
\newtheorem{Prop}[Thm]{Proposition}

\theoremstyle{definition}
\newtheorem{Rem}[Thm]{Remark}

\newtheorem{Def}[Thm]{Definition}
\newtheorem{Obs}[Thm]{Observation}

\newcommand{\A}{\mathcal A}
\newcommand{\wh}{\widehat}
\newcommand{\C}{\mathcal C}
\newcommand{\AC}{\mathcal AC}

\newcommand{\id}{\mathrm{id}}
\newcommand{\inv}{^{-1}}

\def\v{\vskip.12in}
\begin{document}

\title{Heegaard genus and complexity of fibered knots}

\author{Mustafa Cengiz}

\address{Department of Mathematics \\ Boston College \\ Chestnut Hill, Massachusetts, USA}

\email{mustafa.cengiz@bc.edu}

\begin{abstract}
We prove that if a fibered knot $K$ with genus greater than one in a three-manifold $M$ has a sufficiently complicated monodromy, then $K$ induces a minimal genus Heegaard splitting $P$ that is unique up to isotopy, and small genus Heegaard splittings of $M$ are stabilizations of $P$. We provide a complexity bound in terms of the Heegaard genus of $M$. We also provide global complexity bounds for fibered knots in the three-sphere and lens spaces.
\end{abstract}
\thanks{}

\v \v \v

\maketitle

\setcounter{section}{0}


\section{Introduction}
Throughout this paper, $M$ denotes a closed, connected, orientable three-manifold. A \emph{fibered link} in $M$ is an embedded link $L$ such that there is a fibration $p:(M\setminus L)\to S^1$ with fibers, called \emph{pages}, homeomorphic to the interior of a compact surface $\Sigma$. Identifying $S^1$ with $[0,2\pi]/\sim$, we denote the page $p\inv(\theta)$ by $\Sigma_\theta$, and each page has $L$ as its boundary. The \emph{exterior} of a link $L\subset M$, denoted by $X_L$, is the complement of an open tubular neighborhood $\mathring{N}(L)$ in $M$. When $L$ is a fibered link, the restriction of the fibration map $p$ to $X_L$ is still a fibration with fibers homeomorphic to the compact surface $\Sigma$. When we cut $X_L$ open along a fiber, we get an interval bundle homeomorphic to $\Sigma\times [0,2\pi]$. Hence, $X_L$ is homeomorphic to the mapping torus $M_\phi=\Sigma\times [0,2\pi]/(x,0)\sim (\phi(x),2\pi)$ for some homeomorphism  $\phi:\Sigma\to\Sigma$ such that $\phi|_{\partial\Sigma}=\id$. The homeomorphism $\phi$ is called a \emph{monodromy} of the fibered link $L$.

In this paper, we assume that all fibered links (or knots) have pages homeomorphic to $\Sigma$ with Euler characteristic $\chi(\Sigma)\le -3$. In particular, fibered knots have pages of genus greater than or equal to two. The complexity of a monodromy $\phi$ is measured using the arc-and-curve complex of $\Sigma$, which is defined in the following way. A properly embedded arc in $\Sigma$ is called \emph{inessential} if it cuts off a disk from $\Sigma$. A simple closed curve embedded in $\Sigma$ is called \emph{inessential} if either it is trivial (it bounds a disk in $\Sigma$) or it is peripheral (it cuts off an annulus from $\Sigma$). An arc or a curve properly embedded in $\Sigma$ is called \emph{essential} if it is not inessential. Let $Z\subset \partial \Sigma$ be a collection of points, one in each boundary component of $\Sigma$. The \emph{arc-and-curve complex} of $\Sigma$, denoted by $\AC(\Sigma)$, is the abstract simplicial complex of which vertices are isotopy (rel $Z$) classes of essential arcs and curves in $\Sigma$, and $k$-simplices are $k$-tuples of pairwise disjoint (up to isotopy rel $Z$) essential arcs and curves in $\Sigma$. In particular, if two non-isotopic (rel $Z$) essential arcs or curves are disjoint up to isotopy, then their isotopy classes bound an edge in $\AC(\Sigma)$. For simplicity, we do not distinguish an arc or a curve from its isotopy class in notation. The \emph{distance} between two isotopy classes $\gamma_1,\gamma_2$ of essential arcs or curves in $\Sigma$, denoted by $d_{\AC}(\gamma_1,\gamma_2)$,  is then the minimum number of edges between corresponding vertices in the arc-and-curve complex. For a fibered link $L$ with the monodromy $\phi:\Sigma\to\Sigma$, we define the \emph{complexity} of $\phi$ (or $L$) by
\begin{align*}
d_{\AC}(\phi)=\min\{d_{\AC}(\gamma,\phi(\gamma))\mid \gamma \text{ is a vertex in }\AC(\Sigma)\}.
\end{align*}
One defines the arc complex, denoted by $\A(\Sigma)$, and the curve complex, denoted by $\C(\Sigma)$, similarly. The corresponding complexities $d_\A(\phi)$ and $d_\C(\phi)$ are also defined in a similar fashion. It immediately follows that $d_{\AC}(\phi)\le d_{\A}(\phi)$ and $d_{\AC}(\phi)\le d_{\C}(\phi)$.

Saul Schleimer has the following conjecture regarding the complexity of fibered knots in three-manifolds.
\begin{Conj}[Schleimer \cite{Schleimer}, Thompson \cite{Thomp}] \label{Conjecture}
For any three-manifold $M$, there is a constant $t(M)$ with the following property: if $K \subset M$ is a fibered knot, then the monodromy of $K$ has complexity at most $t(M)$. Moreover, $t(S^3) = 1$.
\end{Conj}
For a given three-manifold $M$, one could aim at defining the complexity bound $t(M)$ in terms of the Heegaard genus of $M$. A \emph{Heegaard surface} (or a \emph{Heegaard splitting}) of $M$ is an embedded, closed, separating surface $P$ in $M$ which bounds a pair of handlebodies $(U,V)$ such that $U\cup V=M$ and $U\cap V=\partial U=\partial V=P$. We may also refer to the triple $(P,U,V)$ as a Heegaard splitting. The \emph{Heegaard genus} of $M$, denoted by $g(M)$, is the minimum genus among all Heegaard splittings of $M$.

Any fibered link $L\subset M$ induces a Heegaard surface $P=(\Sigma_0\cup L \cup \Sigma_{\pi})$ bounding the handlebodies  $U=(p\inv([0,\pi])\cup L)$ and $V=(p\inv([\pi,2\pi])\cup L)$  homeomorphic to $\Sigma\times I$ embedded in $M$. It is easy to see that $g(P)=1-\chi(\Sigma)$.

In favor of Schleimer's conjecture, we prove the following theorem in this paper.

\begin{Thm}\label{MainTheorem}
Let $K\subset M$ be a non-trivial fibered knot with monodromy $\phi$ and pages of genus greater than one.
\begin{enumerate}
	\item If $M\cong S^3$, then $d_\A(\phi)\le 3$.
	\item If $M\cong S^1\times S^2$, then $d_\A(\phi)\le 3$.
	\item If $M$ is a lens space, then $d_\A(\phi)\le 4$.
	\item Let $P\subset M$ be a minimal Heegaard surface with genus $g\ge 2$. If $d_{\AC}(\phi)> 2g+2$, then $P$ is induced by $K$ and it is unique up to isotopy.
	\end{enumerate}
\end{Thm}
The techniques we use in the proof of the main theorem are strong enough to provide the following result.
\begin{Thm} \label{StabilizationTheorem}
Let $K\subset M$ be a fibered knot with monodromy $\phi$ and pages of genus greater than one, and $P\subset M$ a minimal Heegaard splitting with genus $g\ge 2$. If $d_{\AC}(\phi)> 2h+2$ for some integer $h>g$, then any non-minimal Heegaard surface $P'\subset M$ with genus $g'\le h$ is a stabilization of $P$.
\end{Thm}

Scharlemann and Tomova \cite{ScharlemannTomova} proved that a sufficiently complicated Heegaard splitting (in terms of Hempel distance) in a three-manifold $M$ is minimal genus and unique up to isotopy in $M$, which was later proved by Li \cite{Li}, using simpler arguments. Moreover, Bachman and Schleimer \cite{BachmanSchleimer} proved that if a closed fibered three-manifold $M$ has sufficiently complicated monodromy, then the Heegaard splitting induced by the fibration of $M$ is minimal and unique up to isotopy. Our work is a confirmation of the same phenomenon in a different setting, as Part (4) of Theorem \ref{MainTheorem} implies if a fibered knot $K$ has a sufficiently complicated monodromy, then the Heegaard splitting induced by $K$ is minimal and unique up to isotopy. This result was previously announced in the unpublished preprint \cite{Johnson} (see Theorem 1) by Jesse Johnson, where the proof was given using an axiomatic thin position argument and Bachman's index theory \cite{Bachman}. Here we provide a more direct proof based on standard thin position and double sweepout arguments. The reader will see in the upcoming sections that our proof techniques are very similar to the ones in \cite{BachmanSchleimer} and \cite{Li}.

We finish the introduction by pointing out that Part (4) of Theorem \ref{MainTheorem} affirms Conjecture \ref{Conjecture} in a certain case.
\begin{Cor}
Schleimer's conjecture holds for fibered knots which do not induce minimal genus Heegaard splittings.
\end{Cor} 

\noindent\textbf{Acknowledgements.} I am grateful to my Ph.D. advisor, Tao Li, for his guidance, support, and for numerous discussions. I would also like to thank Saul Schleimer for helpful communications and for providing feedback on an early draft of the paper. Finally, many thanks to the referee for their careful review and valuable suggestions.

\section{Plan of the paper and proofs of the main theorems}

\noindent\textbf{Notation.}  Throught the paper, we denote the interior of any topological space $S$ by $\mathring{S}$. When $S$ is a knot, link, or surface properly embedded in a three-manifold $M$, $\mathring{N}(S)$ (respectively $N(S)$) denotes an open (respectively closed) tubular neighborhood of $S$ in $M$. \v

In this section, we will list the key results that are proved in the upcoming sections and used in the proof of the main theorem. At the end of the section, we will prove Theorems \ref{MainTheorem} and \ref{StabilizationTheorem}, using the listed results. Throughout the paper, we assume the familiarity of the reader with knot theory, surfaces in three-manifolds, handlebodies, and Heegaard splittings. Standard references are \cite{Hatcher}, \cite{Rolfsen}, and \cite{ScharlemannHS}.

\begin{Def}
	A Heegaard splitting $(P,U,V)$ in $M$ is called
\begin{itemize}
	\item \emph{stabilized} if there are essential disks $D\subset U$ and $E\subset V$ such that $D$ and $E$ intersect exactly once in $P$,
	\item \emph{reducible} if there are essential disks $D\subset U$ and $E\subset V$ such that $\partial D=\partial E$ in $P$,
	\item \emph{weakly reducible} if there are essential disks $D\subset U$ and $E\subset V$ such that $\partial D$ and $\partial E$ are disjoint in $P$.
	\item \emph{strongly irreducible} if $\partial D$ and $\partial E$ intersect in $P$ for any essential disks $D\subset U$ and $E\subset V$, i.e., $(P,U,V)$ is not weakly reducible.
\end{itemize}
A sphere $S$ embedded in $M$ is called \emph{essential} if it does not bound a ball in $M$. A positive genus surface $S$ embedded in $M$ is called \emph{essential} (or \emph{incompressible}) if it has no compressing disks, i.e., for any disk $D$ in $M$ such that $D\cap S = \partial D$, there exists a disk $E$ in $S$ such that $\partial E = \partial D$.
\end{Def}
The interaction between Heegaard splittings and essential surfaces in three-manifolds has been well-studied in the three-manifolds literature (for instance, see \cite{CassonGordon}, \cite{Haken}, and \cite{Wald}). We will use these results in various places throughout the paper.

In Section~\ref{EssentialSurfaces}, we prove the following proposition by analyzing the interaction between essential surfaces embedded in $M$ and the pages of a fibered link in $M$. \v

\noindent\textbf{Proposition~\ref{EssentialBound}.}
\emph{Let $L\subset M$ be a fibered link with monodromy $\phi$.
If $M$ contains an essential sphere, then $d_\A(\phi)\le 3$. If $M$ contains a closed incompressible surface of genus $g>0$, then $d_{\AC}(\phi)\le 2g+2$.}\v

The proposition, when combined with some classical theorems on Heegaard splittings, implies the following.\v

\noindent {\bf Theorem~\ref{WeaklyReducibleBound}.}
{\em
Let $L\subset M$ be a fibered link with monodromy $\phi$. If $M$ contains a genus $g\ge 2$ Heegaard surface $P$, which is weakly reducible but not stabilized, then
$$d_{\AC}(\phi)\le\begin{cases}
\quad\ 3 \,\quad,\text{ if $g=2$,}\\
-\chi(P),\text{ if $g\ge 3$}.
\end{cases}$$
In particular, if a minimal genus Heegaard surface $P\subset M$ is weakly reducible, then the given complexity bound holds.
}\v

In Section~\ref{ThinPosition}, we introduce the thin position and double sweepout techniques. These will be useful to achieve a complexity bound for the monodromy of a fibered knot $K$ that cannot be isotoped into a Heegaard surface $P$. In particular, we prove the following.
\v

\noindent\textbf{Theorem~\ref{ThinTheorem}.} \emph{Let $K\subset M$ be a fibered knot with monodromy $\phi$ and pages of genus greater than one. If $P\subset M$ is a Heegaard surface of genus $g$ such that $K$ cannot be isotoped into $P$, then 
	$$d_\A(\phi)\le\begin{cases}
	\ \ \ \ 3\ \ \ \ \ ,\text{ if $g=0$, }\\
	2g+2\,,\text{ if $g\ge 1$.}
	\end{cases}$$}\v

Proposition \ref{EssentialBound} and Theorem \ref{ThinTheorem} suffice to prove Parts (1), (2), and (3) of Theorem \ref{MainTheorem}. Moreover, Theorems \ref{WeaklyReducibleBound} and \ref{ThinTheorem} suffice to prove Part (4) of Theorem \ref{MainTheorem} when a minimal genus Heegaard surface $P$ is weakly reducible or $K$ cannot be isotoped into $P$. So, in Sections \ref{Non-Primitive} and \ref{Primitive}, we analyze the case that a fibered knot $K\subset M$ lies in a strongly irreducible Heegaard surface $P$, which is the only case left for a complete proof of Theorem \ref{MainTheorem}. Namely, we prove the following statement.\v

\noindent\textbf{Theorem~\ref{KinP}.} \emph{Let $K\subset M$ be a fibered knot with monodromy $\phi$ and pages of genus greater than one. If $(P,U,V)$ is a strongly irreducible Heegaard splitting of genus $g\ge 2$ in $M$ such that $K\subset P$, then at least one of the following holds:
\begin{enumerate}
	\item $P$ is isotopic to the Heegaard surface induced by $K$.
	\item $d_{\C}(\phi)\le 2g-2$.
\end{enumerate}}\v 
We prove Theorem~\ref{KinP} in two cases depending on whether $K\subset P$ is primitive in either $U$ or $V$, or non-primitive. Each case requires different tools, whereas some tools that resolve the non-primitive case come in handy in the primitive case as well. Therefore, we devote Section \ref{Non-Primitive} to the proof of Theorem~\ref{KinP} when $K$ is non-primitive, by using standard arguments from the three-manifolds literature. Finally, Section \ref{Primitive} deals with the proof of Theorem~\ref{KinP} when $K$ is a primitive knot, by using double sweepout arguments. 

We finish this section by proving the main theorems, which readily follow from the results stated above.\v

\noindent\emph{Proof of Theorem~\ref{MainTheorem}.} We prove each statement separately.
\begin{enumerate}
	\item The bound for $S^3$ follows from Theorem \ref{ThinTheorem} because a non-trivial fibered knot cannot be isotoped into a Heegaard sphere in $S^3$.
	\item The bound for $S^1\times S^2$ follows from Proposition \ref{EssentialBound} because $S^1\times S^2$ contains an essential sphere.
	\item Let $K$ be a fibered knot in a lens space. If $K$ can be isotoped into a Heegaard torus $T$, then $A=T\setminus \mathring{N}(K)$ is an incompressible annulus that is not $\partial$-parallel in $X_K$. It follows from Lemma \ref{nonzeroslope} (see below) that $d_\A(\phi)\le 1\le 4$. On the other hand, if $K$ cannot be isotoped into $T$, then by Theorem \ref{ThinTheorem}, we obtain $d_\A(\phi)\le 4$ since $T$ has genus 1.
	\item Let $K\subset M$ be a fibered knot with monodromy $\phi$ and pages of genus greater than one, and $P\subset M$ a minimal Heegaard splitting with genus $g\ge 2$.  Assume that $d_{\AC}(\phi)> 2g+2$. By Theorem \ref{WeaklyReducibleBound}, $P$ is strongly irreducible since it is minimal. By Theorem \ref{ThinTheorem}, $K$ can be isotoped into $P$. Finally, by Theorem \ref{KinP}, $P$ is isotopic to the Heegaard surface induced by $K$. Since $P$ is arbitarily chosen, we deduce that $P$ is the unique minimal genus Heegaard surface in $M$ up to isotopy. \qed
\end{enumerate}\vspace{2mm}
	
\noindent\emph{Proof of Theorem~\ref{StabilizationTheorem}.} Let $K\subset M$ be a fibered knot with monodromy $\phi$ and pages of genus greater than one, and $P\subset M$ a minimal Heegaard splitting with genus $g\ge 2$. Assume that $d_{\AC}(\phi)> 2h+2$ for some integer $h>g$. Let $P'$ be a Heegaard splitting in $M$ with genus $g'$ such that $g<g'\le h$. Then
\begin{enumerate}
	\item  Since $d_{\AC}(\phi)>2h+2\ge 2g'+2$, by the contrapositive Theorem \ref{ThinTheorem}, we deduce that $K$ lies in $P'$ up to isotopy. In other words, $K$ and $P'$ satisfy one of the sufficient conditions of Theorem \ref{KinP}.
	\item Since $d_{\AC}(\phi)>2h+2\ge 2g+2$, it follows from Theorem \ref{MainTheorem} that $K$ induces the minimal splitting $P$. Therefore, $K$ does not induce $P'$.  In other words, $K$ and $P'$ do not satisfy necessary condition (1) of Theorem \ref{KinP}.
	\item Since $d_{\AC}(\phi)>2h+2\ge 2g'+2$, $K$ and $P'$ do not satisfy necessary condition (2) of Theorem \ref{KinP}.
\end{enumerate}
Thus, by the contrapositive of Theorem \ref{KinP}, we deduce that $P'$ is weakly reducible. Finally, by the contrapositive of Theorem \ref{WeaklyReducibleBound}, $d_{\AC}(\phi)> 2g'+2$ implies that $P'$ is stabilized. Since this holds for any genus $g'\le h$, we deduce that $P'$ can be destabilized into the minimal genus Heegaard surface $P$.\qed


\section{Essential Surfaces and the Complexity}\label{EssentialSurfaces}
In this section, we argue that the complexity of a monodromy is bounded when $M$ contains an essential surface. Namely, we will prove the following.
\begin{Prop}\label{EssentialBound}
Let $L\subset M$ be a fibered link with monodromy $\phi$.
If $M$ contains an essential sphere, then $d_\A(\phi)\le 3$. If $M$ contains a closed incompressible surface of genus $g>0$, then $d_{\AC}(\phi)\le 2g+2$.
\end{Prop}

We will prove the proposition at the end of this section after introducing some terminology and tools that are used in the proof. However, first we prove an immediate corollary of the proposition combined with Casson-Gordon's theorem (\cite{CassonGordon}, Theorem 3.1) and Waldhausen's theorem (\cite{Wald}, Theorem 3.1).
\begin{Thm}\label{WeaklyReducibleBound}
	Let $L\subset M$ be a fibered link with monodromy $\phi$. If $M$ contains a genus $g\ge 2$ Heegaard surface $P$, which is weakly reducible but not stabilized, then
	$$d_{\AC}(\phi)\le\begin{cases}
	\quad\ 3 \,\quad,\text{ if $g=2$,}\\
	-\chi(P),\text{ if $g\ge 3$}.
	\end{cases}$$
	In particular, if a minimal genus Heegaard surface $P\subset M$ is weakly reducible, then the given complexity bound holds.
\end{Thm}
\begin{proof}
	Let $(P,U,V)$ be a genus $g\ge 2$ Heegaard splitting in $M$ that is weakly reducible but not stabilized. We have two cases.\vspace{1mm}
	
	\noindent\textbf{Case 1.} $P$ is an irreducible splitting: When $P$ is weakly reducible but irreducible, it follows from Theorem 3.1 in \cite{CassonGordon} that $M$ contains an incompressible surface $S$ of positive genus, which is obtained by compressing the Heegaard surface $P$ at least once in both $U$ and $V$. Hence, $S\subset M$ is an incompressible surface such that $0<g(S)\le g-2$. By Proposition \ref{EssentialBound}, we obtain $$d_{\AC}(\phi)\le 2g(S)+2\le 2(g-2)+2=2g-2=-\chi (P).$$
	
	\noindent\textbf{Case 2.} $P$ is a reducible splitting: When $P$ is reducible but not stabilized, it is a corollary of Theorem 3.1 in \cite{Wald} that $M$ contains an essential sphere. Hence, by Proposition \ref{EssentialBound}, we obtain $d_{\AC}(\phi)\le d_{\A}(\phi)\le 3$. When $g\ge 3$, this particularly implies the desired bound $$d_{\AC}(\phi)\le 3< 2g-2=-\chi(P).$$
	
	Finally, the last statement in the theorem then follows from the fact that a minimal genus Heegaard surface is never stabilized.
\end{proof}

Now we start our discussion towards the proof of Proposition \ref{EssentialBound}. Notice that the proposition is stated not only for fibered knots but also for fibered links. For the rest of this section, fix a fibered link $L\subset M$ with a monodromy $\phi$ and pages $\Sigma_\theta$, for $\theta\in [0,2\pi]$. In the proof of the proposition, we take an essential/incompressible surface $S$ that intersects the fibered link $L$ minimally. It follows that $F=S\cap X_L$ is a meridional incompressible and $\partial$-incompressible surface. Following Theorem 4 in \cite{Thurs}, $F$ can be isotoped so that it has only saddle tangencies to the pages of $L$, and its transversal intersection with any page $\Sigma_\theta$ consists of arcs and curves that are essential in $\Sigma_\theta$. One then can obtain the desired complexity bound by analyzing how the isotopy types of the essential arcs and curves in $\Sigma_\theta\cap F$ change as we travel from $\Sigma_0\cap F$ to $\Sigma_{2\pi}\cap F$, similar to the proof of Theorem 3.1 in \cite{BachmanSchleimer}. However, the discussion in this section is constructed in a more general setting (when $F$ is not necessarily incompressible in $X_L$) so that it will be useful in the upcoming sections. Therefore, we will first introduce three subsections dealing with the boundary slopes on $\partial X_L$, tangencies of $F$ to the pages, and intersections of $F$ with the pages.

\begin{Rem} 
An incompressible and $\partial$-incompressible surface, such as $F$ mentioned above, is called \textit{index-zero} in the terminology of Bachman \cite{Bachman}. A complexity bound in this case follows from Lemmas 4 and 17 in \cite{Johnson}, which uses the same counting argument. Our work in this section resembles \cite{Johnson}. However, we strictly diverge from \cite{Johnson} in the following sections.
\end{Rem}

\subsection{Boundary slopes.} \label{boundaryslopes} In the proof of Proposition  \ref{EssentialBound}, surfaces properly embedded in $X_L$ will play an important role. In particular, when $F\subset X_L$ is a properly embedded surface that has boundary, the isotopy classes of $\partial F$ in $\partial X_L$ will be important. 
Therefore, let us first distinguish those isotopy classes.

Let $K$ be a component of a given fibered link $L\subset M$. On the torus component $\partial X_{K}$ of $\partial X_L$, fix simple closed curves $\mu$ and $\lambda$ such that $\mu$ bounds an essential disk in the solid torus $N(K)$, and $\lambda$ lies in the boundary of a page of $L$. Regarding the pair $(\mu, \lambda)$ as homology generators of $H_1(\partial N(K), \mathbb{Z}) \cong \mathbb{Z}^2$, we can express any simple closed curve $\gamma\subset \partial X_K$ as a \emph{sum}, $a\cdot \mu + b\cdot \lambda$, up to isotopy. It follows that any essential simple closed curve $\gamma \subset \partial X_K$ can be expressed as a \emph{slope} in $\mathbb{Q}\cup \{\infty\}$, by simply mapping $\gamma = a\cdot \mu + b\cdot \lambda$ to $a/b$. This association provides the following definitions:

\begin{itemize}
\item The isotopy type of a simple closed curve that bounds a disk in $N(K)$ is called the \emph{infinity slope} or the \emph{meridional slope} in $\partial X_{K}$, and any other isotopy type is called a \emph{non-meridional slope}. A surface that has non-empty meridional boundary components in $\partial X_L$ is called a \emph{meridional} surface.

\item The \emph{zero slope} is the isotopy type of the simple closed curve $\partial X_{K}$ that lies in the boundary of a page of $L$, and any other isotopy type is called a \emph{non-zero slope}. 

\item An isotopy type of a simple closed curve $\gamma\subset \partial X_{K}$ is called \emph{integral} if $\gamma=a\cdot \mu + \lambda$ for some integer $a$. If a simple closed curve $\gamma\subset \partial X_{K}$ cobounds an annulus $A$ with $K$ in $N(K)$, then $\gamma$ can be expressed as the longitude $\lambda$ that twists along the meridian $\mu$ an integer number of times, say $a$. In particular, $\gamma$ can be expressed as $a\cdot \mu + \lambda$, which means that $\gamma$ is an integral slope.
\end{itemize}

\subsection{Tangencies.} In this subsection, we will discuss how to position a properly embedded surface $F\subset X_L$ nicely with respect to the pages of $L$. We first introduce a definition for surfaces with possibly empty boundary components of non-zero slopes.

\begin{Def} A properly emdedded surface $F\subset X_L$ with (possibly empty) boundary components of non-zero slopes in the components of $\partial X_L$ is said to be \emph{regular} in $X_L$ if
\begin{enumerate}
\item The components of $\partial F$ and $\partial \Sigma_\theta$ are transverse in $\partial X_L$ for each $\theta$;
\item $F$ is transverse to each $\Sigma_\theta$ except for finitely many $\theta_1,\ldots,\theta_m\in [0,2\pi]$;
\item $F$ is transverse to each $\Sigma_{\theta_i}$, $i=1,\ldots,m$, except for a single saddle or center tangency.
\end{enumerate}
When $F$ is regular, the pages which are not transverse to $F$ are called \emph{critical} with respect to $F$. A page that is not critical (i.e., transverse to $F$) is called \emph{non-critical}.
\end{Def}

Note that the definition of a regular surface can be extended for surfaces with boundary components of zero slopes, however, we will be mostly dealing with meridional surfaces in this paper. By standard genericity arguments, every properly embedded surface with  $F\subset X_L$ with (possibly empty) boundary components of non-zero slopes can be isotoped to be regular in $X_L$. To prove Proposition \ref{EssentialBound}, we will analyze the tangencies of a regular surface which intersects pages in essential arcs and curves.


\begin{Def}
Let $F\subset X_L$ be a regular surface with (possibly empty) boundary components of non-zero slopes in $\partial X_L$. A saddle tangency of $F$ to a page $\Sigma_\theta$ is called \emph{inessential} if for any $\epsilon>0$ sufficiently small, the component of $F\cap (\Sigma\times [\theta-\epsilon,\theta+\epsilon])$ that contains the tangency has a boundary component that is a trivial simple closed curve in $\Sigma_{\theta\pm \epsilon}$. If a saddle tangency is not inessential, then it is called \emph{essential}.
\end{Def}

\begin{figure}[h]
	\centering
	\def\svgwidth{0.7\linewidth}
\begingroup%
  \makeatletter%
  \providecommand\color[2][]{%
    \errmessage{(Inkscape) Color is used for the text in Inkscape, but the package 'color.sty' is not loaded}%
    \renewcommand\color[2][]{}%
  }%
  \providecommand\transparent[1]{%
    \errmessage{(Inkscape) Transparency is used (non-zero) for the text in Inkscape, but the package 'transparent.sty' is not loaded}%
    \renewcommand\transparent[1]{}%
  }%
  \providecommand\rotatebox[2]{#2}%
  \newcommand*\fsize{\dimexpr\f@size pt\relax}%
  \newcommand*\lineheight[1]{\fontsize{\fsize}{#1\fsize}\selectfont}%
  \ifx\svgwidth\undefined%
    \setlength{\unitlength}{1060.43684306bp}%
    \ifx\svgscale\undefined%
      \relax%
    \else%
      \setlength{\unitlength}{\unitlength * \real{\svgscale}}%
    \fi%
  \else%
    \setlength{\unitlength}{\svgwidth}%
  \fi%
  \global\let\svgwidth\undefined%
  \global\let\svgscale\undefined%
  \makeatother%
  \begin{picture}(1,0.43443867)%
    \lineheight{1}%
    \setlength\tabcolsep{0pt}%
    \put(0,0){\includegraphics[width=\unitlength,page=1]{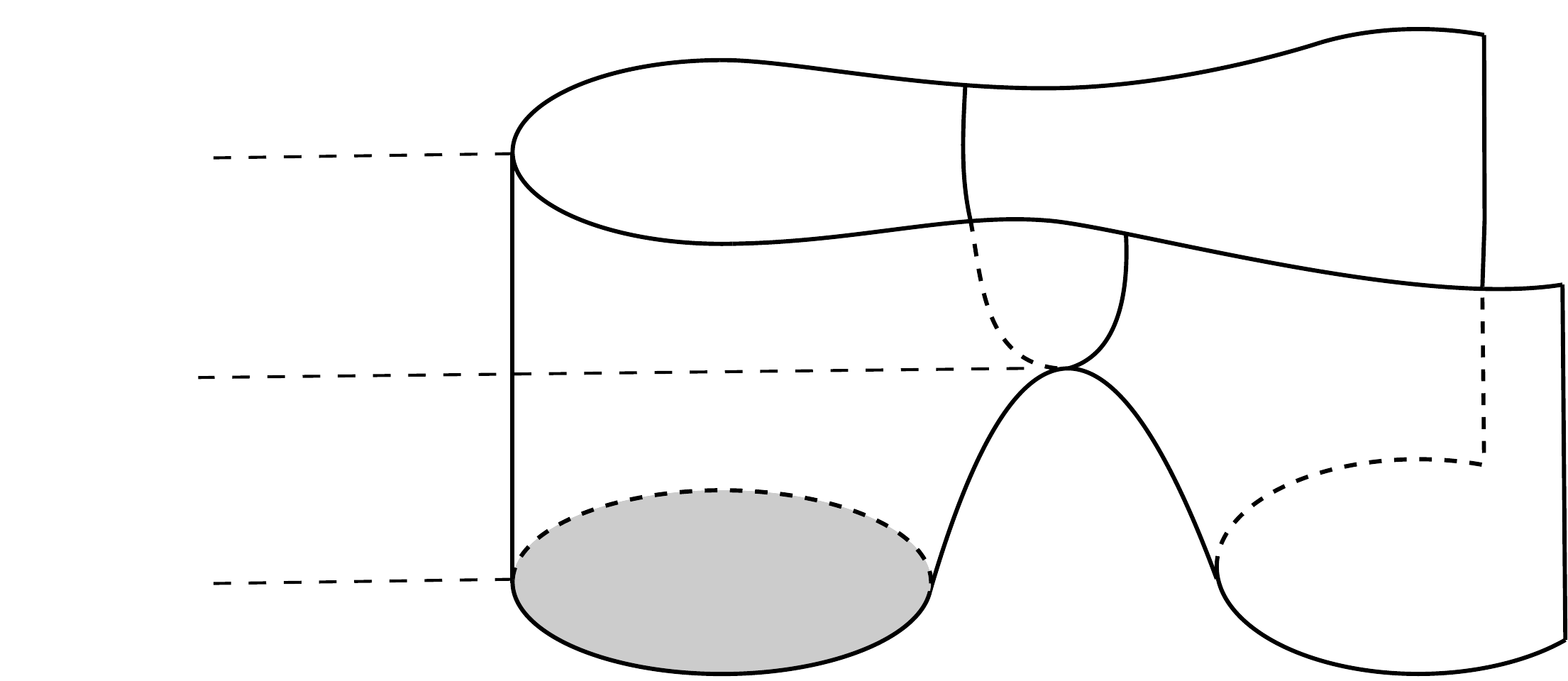}}%
    \put(-0.0021411,0.18911826){\color[rgb]{0,0,0}\makebox(0,0)[lt]{\lineheight{1.25}\smash{\begin{tabular}[t]{l}$\Sigma_\theta$\end{tabular}}}}%
    \put(-0.0020183,0.3289705){\color[rgb]{0,0,0}\makebox(0,0)[lt]{\lineheight{1.25}\smash{\begin{tabular}[t]{l}$\Sigma_{\theta+\epsilon}$\end{tabular}}}}%
    \put(-0.00203321,0.05915458){\color[rgb]{0,0,0}\makebox(0,0)[lt]{\lineheight{1.25}\smash{\begin{tabular}[t]{l}$\Sigma_{\theta-\epsilon}$\end{tabular}}}}%
    \put(0.57987738,0.00498669){\color[rgb]{0,0,0}\makebox(0,0)[lt]{\lineheight{1.25}\smash{\begin{tabular}[t]{l}$C$\end{tabular}}}}%
    \put(0.7763367,0.00724603){\color[rgb]{0,0,0}\makebox(0,0)[lt]{\lineheight{1.25}\smash{\begin{tabular}[t]{l}$\alpha$\end{tabular}}}}%
    \put(0.4301049,0.41266852){\color[rgb]{0,0,0}\makebox(0,0)[lt]{\lineheight{1.25}\smash{\begin{tabular}[t]{l}$\beta$\end{tabular}}}}%
    \put(0.42872577,0.05383874){\color[rgb]{0,0,0}\makebox(0,0)[lt]{\lineheight{1.25}\smash{\begin{tabular}[t]{l}$D$\end{tabular}}}}%
  \end{picture}%
\endgroup%

	\caption{A local picture of an inessential saddle tangency of $F$ to a page $\Sigma_{\theta}$.\label{inessentialsaddle}}
\end{figure}

In Figure \ref{inessentialsaddle}, we depict a local picture of an inessential saddle tangency. Observe that we can project the simple closed curve $C$ and the curves (or arcs) $\alpha$ and $\beta$ in to the page $\Sigma_\theta$  through $F$. After the projections, $\alpha$ and $\beta$ are isotopic in $\Sigma_\theta$ through a copy of the disk $D$ bounded by $C$. In other words, an inessential saddle tangency does not change the isotopy types of essential arcs or curves in $\Sigma_{\theta\pm \epsilon}$. Therefore, in the proof of Proposition \ref{EssentialBound}, inessential saddle tangencies will be negligible, and the number of essential saddles will be important for complexity calculations. We prove the following two lemmas to provide an upper bound for the number of essential saddle tangencies of a regular surface $F \subset X_L$.

\begin{Lem}\label{trivial}
Let $\Sigma_\theta$ be a page of $L$. If $\alpha\subset \Sigma_{\theta}$ is a simple closed curve that bounds a disk in $X_L$, then $\alpha$ bounds a disk in $\Sigma_{\theta}$.
\end{Lem}
\begin{proof}
Let $D\subset X_L$ be an embedded disk bounded by $\alpha$ such that $D$ intersects $\Sigma_{\theta}$ transversely, and $D$ intersects $\Sigma_{\theta}$ minimally among all such disks in $X_L$. It suffices to show that the interior of $D$ is disjoint from $\Sigma_{\theta}$, as $\Sigma_{\theta}$ is incompressible. Assume for a contradiction that $\mathring{D}$ is not disjoint from $\Sigma_\theta$ and pick a simple closed curve $\beta\subset D\cap\Sigma_{\theta}$ that is innermost in $D$ so that $\beta$ bounds a subdisk $\Delta$ in $D$ whose interior is disjoint from $\Sigma_{\theta}$. Since $\Sigma_{\theta}$ is incompressible, $\beta$ bounds a disk $\Delta'$ in $\Sigma_{\theta}$. If $\alpha$ lies in $\Delta'$, then $\alpha$ is trivial in $\Sigma_{\theta}$, so we can assume that $\alpha$ does not lie in $\Delta'$. It follows that $\Delta\cup\Delta'$ forms a sphere that bounds a ball in $X_L$ and we can isotope $D$ through this ball to eliminate $\beta$ (and other curves in $\Delta'$) from the intersection of $D$ with $\Sigma_{\theta}$, while maintaining $\alpha=\partial D$. This contradicts the minimality assumption on $D$.
\end{proof}

\begin{Lem}\label{saddles} If $F\subset X_L$ is a regular surface with $\chi(F)\le 0$, then the number of essential saddle tangencies of $F$ to the pages is at most $|\chi(F)|=-\chi(F)$.
\end{Lem}
\begin{proof} Let $c$ denote the number of center tangencies, $s$ the number of inessential saddle tangencies, and $s'$ the number of essential saddle tangencies of $F$ to the pages. Since each center tangency contributes 1, and each saddle tangency contributes $-1$ to the Euler characteristic of $F$, we have
$$\chi(F)=c-(s+s') \implies s'= -\chi(F)+(c-s).$$
To prove that $s'\le -\chi(F)$, we will show that $(c-s)$ is non-positive, equivalently $s \ge c$. We will do this by analyzing the singular foliation, say $\mathcal{F}$, of $F$ defined by its intersections with the pages. Note that we regard every arc or curve $\alpha$ that is in the intersection of $F$ with a page $\Sigma_\theta$ as a leaf of $\mathcal{F}$, while $\alpha$ is a subset of $F$. So, we write $\alpha\in \mathcal{F}$ and $\alpha\subset F$.

If $\mathcal{F}$ has no leaf that is a trivial simple closed curve in $F$, then there is no center tangency of $F$ to the pages, i.e., $c=0$, and $s\ge c$ trivially holds. So, we can assume that $\mathcal{F}$ has leaves that are trivial simple closed curves in $F$. Then there exists a collection $\mathcal C = \{C_1,\ldots,C_k\}\subset \mathcal{F}$ of simple closed curves, which are pairwise non-isotopic in $F$, such that
\begin{itemize}
\item for each $j\in\{1,\ldots,k\}$,   $C_j$ bounds a disk $D_j$ in $F$, and
\item the collection $\mathcal{C}$ is outermost and maximal in $F$, i.e., if there exists another curve $C'\in \mathcal{F}$ that is trivial in $F$, then there exists a $j\in\{1,\ldots, k\}$ such that either $C'$ is in $D_j$ or $C'$ and $C_j$ cobound an annulus $F$ that is transverse to the pages.
\end{itemize}
For each $j=1,\ldots,k$, let $a_j$ (resp. $b_j$) be the number of center (resp. saddle) tangencies in $D_j$. It follows that 
$$\chi(D_j)=a_j-b_j=1 \implies b_j= 1-a_j.$$
The lemma will follow from the following observations.
\begin{enumerate}
	\item Outside $\cup_{j=1}^k D_j$ there are no center tangencies: This is because the collection $\mathcal{C}$ is maximal.
	\item For $j=1,\ldots, k$, each saddle tangency in $D_j$ is inessential: Let $p$ be a saddle tangency of $F$ to a page $\Sigma_{\theta}$ that lies in some $D_j$. For $\epsilon>0$ sufficiently small, the component of $F\cap (\Sigma\times [\theta-\epsilon,\theta+\epsilon])$ that contains the tangency has boundary components that lie in $D_j$. By the previous lemma, those boundary components are trivial in $\Sigma_{\theta\pm \epsilon}$ since they bound disks in $D_j$. Hence, by definition, $p$ is an inessential saddle tangency.
	\item For $j=1,\ldots, k$, each $C_j$ is trivial in the corresponding page (by the previous lemma) because $C_j$ bounds a disk $D_j$.
	\item For $j=1,\ldots, k$, each $C_j$ meets a different inessential saddle tangency outside $\cup_{j=1}^k D_j$: Otherwise, two $C_j$'s merge at the same inessential saddle tangency, which yields a third curve $C'$ that is trivial in $F$ and that is not contained in $\cup_{j=1}^k D_j$. Notice that $C'$ cannot cobound an annulus with any $C_j$, which contradicts the maximality assumption on the collection $\C$.
\end{enumerate}
By (1), we have $c=a_1+\ldots+a_k$. By (2), the number of inessential saddle tangencies inside $\cup_{j=1}^k D_j$ is $b_1+\ldots +b_k$. By (3) and (4), the number of inessential saddle tangencies outside $\cup_{j=1}^k D_j$ is at least $k$ (one for each $C_j$, and possibly more outside $\cup_{j=1}^k D_j$). Therefore, we obtain $$s \ge k + b_1+ \ldots +b_k = k + (a_1-1) + \ldots +(a_k-1) = a_1+\ldots+a_k = c,$$ as desired.
\end{proof}

\subsection{Perfect surfaces in fibered link exteriors}
In the previous subsection, we argued that every properly embedded surface $F\subset X_L$ with (possibly empty) boundary components of non-zero slopes can be isotoped to be regular in $X_L$ with certain tangency properties. To prove Proposition \ref{EssentialBound}, we will need a regular surface in $X_L$ with an additional nice property that we will call \emph{perfectness}. In this section, we define that property and show that most incompressible surfaces satisfy it.

\begin{Def}[perfectness]\label{perfect} Let $F\subset X_L$ be a properly embedded regular surface.
\begin{itemize}
\item If $F$ has non-empty boundary components of non-zero slopes, then it is called \emph{perfect} if for any non-critical page $\Sigma_\theta$, every arc in $F\cap \Sigma_\theta$ is essential in $\Sigma_\theta$.
\item If $F$ is a closed surface, then it is called \emph{perfect} if for any non-critical page $\Sigma_\theta$, there exists a simple closed curve in $F\cap \Sigma_{\theta}$ that is essential in $\Sigma_{\theta}$.
\end{itemize}
\end{Def}

Even though every surface in $X_L$ is regular up to isotopy, there are surfaces that are not perfect in $X_L$. A trivial example is a $\partial$-parallel annulus in $X_L$. However, we can show that incompressible surfaces, which have non-empty boundary and are not boundary-parallel, are perfect in $X_L$. This will be useful to prove Proposition \ref{EssentialBound} for essential/incompressible surfaces that cannot be isotoped into $\partial X_L$.

\begin{Lem}\label{boundaryperfect}
Let $F\subset X_L$ be an orientable regular surface with non-empty boundary components of non-zero slopes in $\partial X_L$. If $F$ is incompressible, then it is either perfect or a $\partial$-parallel annulus.
\end{Lem}
\begin{proof}
Assume $F$ is not perfect. We will show that it is a $\partial$-parallel annulus in $X_L$. Let $\Sigma_{\theta}$ be a page such that $F\cap \Sigma_\theta$ contains an arc that is inessential in $\Sigma_\theta$. Then there exists an arc $\alpha\subset F\cap\Sigma_{\theta}$ that cuts off a half disk $\Delta\subset \Sigma_{\theta}$ whose interior does not intersect $F$. Since $\partial F$ and $\partial \Sigma_\theta$ are transverse, $\partial \Delta$ meets two distinct boundary components of $\partial F$, which cobound an annulus $A$ in a component of $\partial X_L$. Let $N(\Delta) \cong \Delta\times [-1,+1]$ represent a neighborhood of $\Delta$ in $X_L$, and $\Delta^\pm$ denote $\Delta\times \{\pm1\}$. Joining $\Delta^+$ and $\Delta^-$ with the band $B = A\setminus N(D)$, we obtain a compressing disk $D = \Delta^+\cup B \cup \Delta^-$ in $X_L$ such that $D \cap F=\partial D$. Since $F$ is incompressible, we deduce that $\partial D$ bounds a disk in $F$, and thus, $F$ is an annulus. By construction, $F$ is $\partial$-parallel.
\end{proof}
A similar lemma holds for incompressible surfaces in $M$ that completely lie in $X_L$. However, to prove the lemma, we will need the following operation.\vspace{2mm}

\noindent \textbf{Annulus Surgery.} Let $F$ be any closed surface in a fibered link exterior $X_L$. Assume that $F$ is transverse to a page $\Sigma_\theta$ such that a curve $\gamma\subset F\cap \Sigma_\theta$ cuts off an annulus $A\subset \Sigma_\theta$ whose interior is disjoint from $F$. Take a neighborhood $N(A)\cong A\times [-1,1]$ of $A$ in $X_L$ so that $N(A)\cap F=N(\gamma)\cong \gamma\times[-1,1]\subset F$. Now let $T$ be the boundary component of $\partial X_L$ which meets $A$, and $A'$ the annulus $T\setminus N(A)$. It follows that $(N(\gamma)\cup A^-\cup A' \cup A^+)$ is a peripheral torus isotopic to $T$ in $X_L$, where $A^\pm$ represents $A\times \{\pm 1\}$ in $N(A)$. Isotoping $F$ in $M$ through the solid torus bounded by $T$ in $M$, we can replace the annulus $N(\gamma)$ by the annulus $(A^-\cup A' \cup A^+)$. This operation is called the \emph{annulus surgery} of $F$ along $\gamma$. In Figure \ref{ansurg}, we describe the isotopy in a schematic picture, where the link $L$ and the curve of intersection $\gamma$ are represented as dots. Notice that this isotopy of $F$ in $M$ eliminates  $\gamma$ from $F\cap \Sigma_\theta$.
\begin{figure}[h]
	\centering
	\def\svgwidth{0.7\linewidth}
	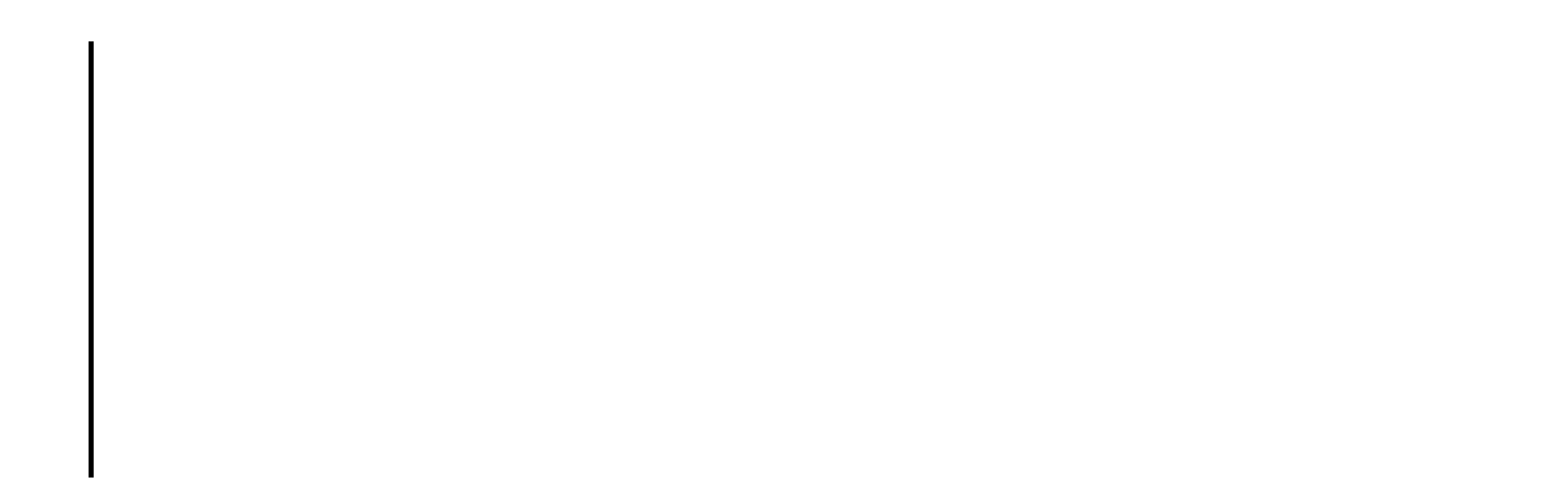
	\caption{The result of annulus surgery is on the right.\label{ansurg}}
\end{figure}

\noindent The annulus surgery is useful to establish the following lemma.
\begin{Lem}\label{closedperfect}
Let $F$ be a closed incompressible surface in $M$. If $F$ is disjoint from a fibered link $L$ and regular in $X_L$, then $F$ is perfect in $X_L$.
\end{Lem}
\begin{proof}
We prove the contrapositive. Assume that $F$ is not perfect. Let $\Sigma_\theta$ be a page transverse to $F$ such that every curve in $F\cap \Sigma_\theta$ is inessential in $\Sigma_\theta$. If there are trivial curves in the intersection, then by applying the standard innermost curve argument, we can isotope $F$ to eliminate those trivial curves from the intersection. So, we can assume that every curve in $F\cap\Sigma_{\theta}$ is peripheral in $\Sigma_{\theta}$. Applying repeated annulus surgeries to $F$ along the peripheral curves, starting with outermost ones, we can further isotope $F$ in $M$ to eliminate all peripheral curves from $F\cap \Sigma_{\theta}$. At the end, we get an isotopic copy of $F$ in $M$ that is disjoint from the page $\Sigma_{\theta}$. In other words, $F\subset M$ can be isotoped into the handlebody $X_L\setminus \mathring{N}(\Sigma_\theta)$, which implies that $F$ is compressible in $M$.
\end{proof}

\subsection{Complexity bounds}
In the proof of Proposition \ref{EssentialBound}, we will introduce an incompressible surface in $X_L$, which is also perfect. Therefore, in this subsection, we observe how a perfect surface in $X_L$ imposes a complexity bound for the  monodromy $\phi$ of $L$. The following lemma is adapted from Lemma 17 in \cite{Johnson} with an improvement on the upper bound.

\begin{Lem}\label{meridional}
Let $F\subset X_L$ be a genus $g$ perfect surface with non-empty meridional boundary components in $\partial X_L$ such that $|\partial F| =  2n$. Then
$$d_\A(\phi)\le\begin{cases}
\ \ \ \ 0\ \ \ \ \, ,\text{ if $g=0$ and $n=1$,}\\
\ \ \ \ 3\ \ \ \ \, ,\text{ if $g=0$ and $n\ge 2$,}\\
 \ \ \ 2g\ \ \ \ ,\text{ if $g\ge 1$ and $n=1$,}\\
2g+2,\text{ if $g\ge 1$ and $n\ge 2$.}
\end{cases}$$
\end{Lem}
\begin{proof}
Let $S$ be the preimage of $F$ under the quotient map $q: \Sigma\times [0,2\pi]\to X_L$, which maps $\Sigma\times\{\theta\}$ to $\Sigma_\theta$ in the natural way. For simplicity, we will not distinguish $\Sigma\times\{\theta\}$ and $\Sigma_\theta$. We can assume that $\Sigma_0$ is transverse to $F$ (after slightly rotating the pages, if necessary). Since $|\partial F|=2n$, there exist $n$ arcs in $S\cap \Sigma_0$, and $S\cap (\partial \Sigma\times [0,2\pi])$ consists of vertical arcs $\{x_1,\ldots,x_{2n}\}\times [0,2\pi]$ as the boundary components are meridional.

By Lemma \ref{saddles}, the number, say $m$, of essential saddle tangencies of $F$ to the pages is at most $-\chi(F)=2g+2n-2$. Let $0<\theta_1<\ldots<\theta_m<2\pi$ be the angles such that each essential saddle tangency of $F$ is contained in one of the $\Sigma_{\theta_i}$'s. For $i=1,\ldots,m-1$, pick $\theta_i<t_i<\theta_{i+1}$ such that $\Sigma_{t_i}$ is transverse to $F$. Furthermore, set $t_0=0$ and $t_m=2\pi$.

Now we will investigate how many different isotopy classes of arcs in $\Sigma$ can be observed in the intersection of $S$ with $\Sigma_\theta$, $\theta \in[0,2\pi]$. To begin with, there are $n$ essential arcs in $S\cap\Sigma_0$. Moreover, for any pair $\theta<\theta'$ in $[0,2\pi]$, we have the following observations:
\begin{enumerate}
\item[(a)] If $\Sigma\times[\theta,\theta']$ contains no essential saddle tangencies of $F$, then the arcs in $S\cap \Sigma_{\theta}$ and $S\cap \Sigma_{\theta'}$ represent the same isotopy types in $\Sigma$.
\item[(b)] If $\Sigma\times[\theta,\theta']$ contains a single essential saddle tangency of $F$, then at most two isotopy classes of arcs in $S\cap \Sigma_{\theta'}$ are different from the isotopy classes of arcs in $S\cap \Sigma_{\theta}$. Different isotopy classes are introduced as two arcs (or an arc and a simple closed curve) in $\Sigma_\theta$ merge at the essential saddle tangency, or vice versa.
\end{enumerate}
These observations imply that as $\theta$ increases from $0$ to $2\pi$,  $S\cap \Sigma_\theta$ realizes at most two new arc types in $\Sigma_\theta$ exactly when $\theta$ passes through one of the $\theta_i$'s, $i=1,\ldots,m$. It follows that the number, say $N$, of isotopy classes of arcs that can be observed in $S\cap \Sigma_{\theta}$, for $\theta\in [0,2\pi]$, is at most 
$$n+2m \le n+2[-\chi(F)]= n+2(2g+2n-2)=4g+5n-4.$$
Note also that each of these arcs is essential in its respective page because $F$ is assumed to be perfect. Moreover, every arc type that is realized in $S\cap \Sigma_{\theta}$ has its endpoints in $\{x_1,\ldots,x_{2n}\}\subset \Sigma_\theta$ by abusing the notation. For $j=1,\ldots,2n$, let $k_j$ be the number of isotopy classes of arcs in $F\cap\Sigma_{\theta}$ that have $x_j$ as an endpoint. It follows that $$k_1+\ldots+k_{2n}=2N\le10n+8g-8$$ 
because each of the $N$ isotopy classes is counted twice (once for each endpoint) in the sum. We deduce that for some $x_j$, the number $k_j$ is at most $(10n+8g-8)/2n=5+(4g-4)/n$.  Without loss of generality, let $x_1$ be the endpoint realized by
$$k\le 5+(4g-4)/n$$ isotopy classes of arcs. Let $\alpha_1,\ldots,\alpha_{k}$ be those isotopy classes. For $i=1,\ldots,k-1$, up to relabelling, we can assume that $\alpha_{i+1}$ is introduced as $\alpha_{i}$ merges into a saddle tangency of $F$. Therefore, $\alpha_{i}$ and $\alpha_{i+1}$ can be isotoped to be disjoint in $\Sigma$, i.e., $d_\A(\alpha_{i},\alpha_{i+1})\le 1$. Moreover, $\phi(\alpha_1)=\alpha_{k}$ because $\phi(x_1)=x_1$. Thus, we obtain
$$d_\A(\phi)\le d_\A(\alpha_1,\phi(\alpha_1))=d_\A(\alpha_1,\alpha_{k})\le\sum_{i=1}^{k-1} d_\A(\alpha_i,\alpha_{i+1})=k-1.$$

Now we will run a case analysis depending on the values of $g$ and $n$ to provide the complexity bounds stated in the lemma.
\vspace{2mm}
	
\noindent\textbf{Case 1.} $g=0$ and $n=1$: In this case, $F$ is an annulus and has no essential saddle tangencies to the pages since $\chi (F)=0$, and we get $k=1$. Hence, we obtain $d_\A(\phi)\le k-1=0$, which implies  $d_\A(\phi)=0$.
\vspace{2mm}

\noindent\textbf{Case 2.} $g=0$ and $n\ge 2$: In this case, we have $k\le 5+(4g-4)/n=5-4/n$, which implies that $k$ is at most $4$. Thus, we obtain
$d_\A(\phi)\le k-1\le 3.$\vspace{2mm}

\noindent\textbf{Case 3.} $g\ge 1$ and $n=1$: In this case, $k\le 5+4g-4=4g+1$. However, $4g+1$ is an unnecessarily large upper bound. Because when $n=1$, we have a single essential arc observed by $F\cap \Sigma_{\theta}$ in between each pair of consecutive essential saddle tangencies. Hence, $k\le 1+m \le 1+2+2g-2=2g+1$. Thus, we obtain $d_\A(\phi)\le k-1\le 2g+1-1=2g$.\vspace{2mm}

\noindent\textbf{Case 4.} $g\ge 1$ and $n\ge 2$: In this case, $k\le 5+(4g-4)/n\le 2g+3$. Thus, we obtain $d_\A(\phi)\le k-1\le 2g+3-1=2g+2$.
\end{proof}

We will now prove a similar lemma for closed perfect surfaces in $X_L$.
\begin{Lem}\label{closedbound}
Let $F\subset X_L$ be a properly embedded, closed, perfect surface. If $F$ is a torus, then $d_\C(\phi)\le 1$. If $F$ has genus $g\ge 2$, then $d_\C(\phi)\le -\chi(F)$.
\end{Lem}
\begin{proof}
Let $S$ be the preimage of $F$ under the quotient map $q: \Sigma\times [0,2\pi]\to X_L$ mapping $\Sigma\times\{\theta\}$ to $\Sigma_\theta$ in the natural way. For simplicity, we will not distinguish $\Sigma\times\{\theta\}$ from $\Sigma_\theta$. We can assume that $\Sigma_0$ is transverse to $F$ (after slightly rotating the pages, if necessary).
	
Assume $F$ is a torus. Since $F$ is perfect, there exists a simple closed curve $\alpha\subset S\cap \Sigma_0$ that is essential in $\Sigma_0$. Notice that there are no essential saddle tangencies in $F$ because $\chi(F)=0$. Therefore, the isotopy type of $\alpha$ does not change at all in $S\cap \Sigma_{\theta}$ from $\Sigma_0$ to $\Sigma_{2\pi}$. Hence, there exists a curve $\beta\subset S\cap \Sigma_{2\pi}$ that is isotopic to $\alpha\subset S\cap \Sigma_{0}$ (when both are regarded in $\Sigma$). Since $F$ is a properly embedded surface in $X_L$, it follows that either $\beta=\phi(\alpha)$ or $\beta\cap \phi(\alpha)=\emptyset$. Respectively, we get either $\alpha=\phi(\alpha)$ or $\alpha\cap \phi(\alpha)=\emptyset$ up to isotopy. In both cases, we obtain $d_\C(\alpha,\phi(\alpha))\le 1$, and hence $d_\C(\phi)\le 1$.

Now assume that $F$ has genus $g\ge 2$. Note that the isotopy classes of essential simple closed curves in $F\cap \Sigma_\theta$ can only change at essential saddle tangencies. Let $0<\theta_1<\ldots<\theta_m<2\pi$ be the angles such that each essential saddle tangency of $F$ is contained in one of the $\Sigma_{\theta_i}$'s. By Lemma \ref{saddles}, we have $m \le -\chi(F)$. For $i=1,\ldots,m-1$, pick $\theta_i<t_i<\theta_{i+1}$ such that $\Sigma_{t_i}$ is transverse to $F$. Since $F$ is perfect, each transversal intersection $S\cap \Sigma_{t_i}$ contains a simple closed curve, say $\alpha_i$, that is essential in $\Sigma_{t_i}$. Furthermore, pick a simple closed curve $\alpha_0\subset S\cap \Sigma_0$ that is essential in $\Sigma_0$ and set $\alpha_m=\phi(\alpha_0)$ in $S\cap \Sigma_{2\pi}$. For $i=0,\ldots,m$, we have the following observations:
\begin{enumerate}
	\item[(a)] If $\alpha_i$ and $\alpha_{i+1}$ are in the boundary of the component of $S\cap (\Sigma\times [t_i,t_{i+1}])$ that contains the essential saddle tangency of $F$ to $\Sigma_{\theta_i}$, then $\alpha_{i+1}$ is introduced as $\alpha_i$ merges into the saddle tangency, or vice versa. Therefore, $\alpha_{i}$ and $\alpha_{i+1}$ can be isotoped to be disjoint in $\Sigma$, i.e., $d_\C(\alpha_{i},\alpha_{i+1})\le 1$. 
	\item[(b)] If one of $\alpha_{i}$ and $\alpha_{i+1}$ is not in the boundary of the component of $S\cap (\Sigma\times
	 [t_i,t_{i+1}])$ that contains the essential saddle tangency of $F$ to $\Sigma_{\theta_i}$, then it is observed in 
	 both $\Sigma_{t_i}$ and $\Sigma_ {t_{i+1}}$ since it is not
	 affected by the essential saddle tangency.  Therefore, $\alpha_{i}$ and $\alpha_{i+1}$ can be isotoped to be disjoint in $\Sigma$, i.e., $d_\C(\alpha_{i},\alpha_{i+1})\le 1$. 
\end{enumerate}
It immediately follows from the observations that 
$$d_\C(\phi)\le d_\C(\alpha_0,\phi(\alpha_0))=d_\C(\alpha_0,\alpha_m)\le \sum_{i=0}^{m-1}d_\C(\alpha_i,\alpha_{i+1})\le m=-\chi(F),$$
as desired.
\end{proof} 

Now we are ready to prove that the existence of an essential/incompressible surface in $M$ imposes an upper bound on the complexity of fibered links.\vspace{2mm}

\noindent\textit{Proof of Proposition~\ref{EssentialBound}.} Let $S\subset M$ be a closed essential surface of genus $g$ that intersects the fibered link $L$ transversely and minimally among all genus $g$ essential surfaces embedded in $M$. We have two cases.\vspace{2mm}

\noindent\textbf{Case 1.} $S\cap L$ is non-empty: In this case, $F=S\cap X_L$ is a properly embedded genus $g$ surface with meridional boundary components in $\partial X_L$ since $S$ intersects $L$ transversely. It follows from standard arguments that $F$ is incompressible in $X_L$ and it is not a $\partial$-parallel annulus.
By Lemma \ref{boundaryperfect}, $F$ is perfect in $X_L$. By Lemma \ref{meridional}, we obtain $d_\A(\phi)\le 3$ when $g(F)=g(S)=0$, and $d_\A(\phi)\le 2g+2$ when $g(F)=g(S)$ is positive.\vspace{2mm}

\noindent\textbf{Case 2.} $S\cap L$ is empty: In this case, $S$ cannot be a sphere since there exists no essential sphere in a fibered link exterior. By Lemma \ref{closedperfect}, $S$ is a perfect surface in $X_L$. It follows from Lemma \ref{closedbound} that $d_\C(\phi)\le \max \{1,2g-2\}\le 2g+2$.\v

In either case, we obtain the desired upper bounds for $d_{\AC}(\phi)$, and this completes the proof.\qed

\section{Fibered Knots in Thin Position}\label{ThinPosition}
In this section, we will prove a theorem which provides the complexity bound stated in Theorem \ref{MainTheorem} when a fibered knot $K$ does not lie on a Heegaard splitting up to isotopy. Namely, we will prove the following.

\begin{Thm}\label{ThinTheorem} Let $K\subset M$ be a fibered knot with monodromy $\phi$ and pages of genus greater than one. If $P\subset M$ is a Heegaard surface of genus $g$ such that $K$ cannot be isotoped into $P$, then 
	$$d_\A(\phi)\le\begin{cases}
	\ \ \ \ 3\ \ \ \ \ ,\text{ if $g=0$, }\\
	2g+2\,,\text{ if $g\ge 1$.}
	\end{cases}$$
\end{Thm}

In the previous section, our assumptions were strong enough to provide a meridional incompressible and perfect surface in the fibered knot exterior, which helped us execute a combinatorial argument that gives a complexity bound on the monodromy. However, there exist fibered knots which contain no incompressible surfaces in their exterior (namely, small knots). In this section, we use thin position and double sweepout arguments to provide a meridional surface that behaves similarly to perfect surfaces. Such a surface will reveal itself as a level surface for a thin position of $K$ with respect to a sweepout of the Heegaard surface $P$ (see below for definitions). The techniques we use here are similar to those in \cite{BachmanSchleimer}, \cite{GordonLuecke}, and \cite{Li}. Before proving the theorem, we will introduce literature, notation, and some useful lemmas. The proof of the theorem will be presented at the end of this section. \vspace{2mm}

\noindent\textbf{Assumption.} For the rest of this section, assume that $K\subset M$ is a fibered knot with monodromy $\phi$, which cannot be isotoped into the given Heegaard surface $P$.

\subsection{Sweepouts of Heegaard splittings.} A \emph{spine} of a handlebody $U$ is a connected graph $G$ in $U$ such that $U\setminus G$ is homeomorphic to $\partial U\times (0,1]$. Let $P$ be a Heegaard surface bounding a pair of hendlebodies $(U,V)$ in $M$. Let $G_U$ and $G_V$ be spines of $U$ and $V$, respectively. Then $M\setminus (G_U\cup G_V)$ is homeomorphic to $P\times (0,1)$. A \emph{sweepout} of the Heegaard surface $P$ is a smooth function $H: P\times I\to M$ such that $H(P\times \{0\})= G_U$, $H(P\times \{1\})= G_V$, and $H(P\times \{t\})$ is isotopic to $P$ for any $t\neq 0,1$. For simplicity, we will denote $H(P\times \{t\})$ by $P_t$, and we will not distinguish $H(P\times(0,1))$ from $P\times (0,1)$. On the other hand, a \emph{height function} of $P$ is the map from $h:P\times (0,1)\to (0,1)$, which maps $P_t$ to $t$.

\subsection{Thin position.} Thin position for knots was invented by Gabai \cite{Gabai} and applied in many places in the three-manifolds literature. For convenience, we recall the definition of a thin position. Fix a sweepout of $P$ in $M$ with height function $h$. By an isotopy of $K$, we may assume that $K\cap (G_U\cup G_V)=\emptyset$ and that $h|_K$ is a Morse function, i.e., $h|_K$ has only finitely many non-degenerate critical values $a_1,\ldots, a_n$ such that $K$ has a unique tangency to each $P_{a_i}$. Given such a Morse position of $K$, let $t_1,\ldots,t_{n-1}\in (0,1)$ be non-critical values of $h|_K$ such that $a_i<t_i<a_{i+1}$ for each $i=1,\ldots,n-1$. We call the number $\Sigma_{i=1}^{n-1} |P_{t_i}\cap K|$ the \emph{width} of the Morse position. A \emph{thin position} of $K$ with respect to the Heegaard splitting $P$ is then a Morse position with the minimal width. In a thin position of $K$ with respect to the Heegaard surface $P$, for each non-critical value $t$ of $h|_K$, the Heegaard surface $P_t$ intersects a tubular neighborhood $N(K)$ of $K$ in meridional disks. In other words, $F_t=P_t\setminus \mathring{N}(K)$ is a meridional surface in $X_K$ and we call it a \emph{level surface}.

Now let $a<b$ in $(0,1)$ such that $a$ is a local minimum of $h|_K$, $b$ is a local maximum of $h|_K$, and $(a,b)$ contains no critical values of $h|_K$. The family $\{F_t\,|\,t\in (a,b)\}$ is called a \emph{middle slab}. We will analyze the intersection of the pages $\Sigma_\theta$ with the levels $F_t$ in a middle slab to introduce a useful meridional level surface $F_s$ in $X_K$. \vspace{2mm}

\noindent\textbf{Assumption.} For the rest of this section, assume that $K$ is in thin position with respect to a fixed sweepout $\{P_t\,|\,t\in [0,1]\}$ of the given Heegaard surface $P$ and fix a middle slab $\{F_t\,|\,t\in (a,b)\}$. 

\subsection{Intersection graphics of surface families} \label{Cerf}
By standard generecity arguments, one can isotope the pages $\Sigma_\theta$ in $X_K$ so that they are \emph{standard} with respect to level surfaces $F_t$ near $\partial X_K$, i.e.,
\begin{itemize}
	\item there exists a page $\Sigma_{\theta}$ which hangs up (respectively hangs down) near the local minimum $a$ (respectively the local maximum $b$) of $K$ with respect to $P_t$, and
	\item away from the local minimums and local maximums of $K$, the pages $\Sigma_{\theta}$ are transverse to $\partial X_K$.
\end{itemize}
For local pictures that represent the conditions above, see Figures 4 and 5 in \cite{GordonLuecke}. Furthermore, by Cerf theory \cite{Cerf}, the pages $\Sigma_\theta$ can be further isotoped so that the pages $\Sigma_{\theta}$ and the level surfaces $F_t$ of the middle slab are in \emph{Cerf position}, that is, the set
$$\Lambda=\{(\theta,t)\in S^1\times (a,b)\,|\,\Sigma_\theta\text{ is not transverse to }F_t\}$$
is a one-dimensional graph in the open annulus $A=S^1\times (a,b)$ satisfying the following properties:
\begin{enumerate}
\item If $(\theta,t)$ is in the complement of $\Lambda$, then $\Sigma_\theta$ and $F_t$ intersect transversely in a collection of properly embedded arcs and simple closed curves (by definition of $\Lambda$).
\item If $(\theta,t)$ and $(\theta',t')$ are in the same connected component of $A\setminus \Lambda$, then $\Sigma_\theta\cap F_t$ and $\Sigma_{\theta'}\cap F_{t'}$ have the \emph{same} intersection pattern, i.e., the intersections $\Sigma_\theta\cap F_t$ and $\Sigma_{\theta'}\cap F_{t'}$ are isotopic in a copy of page $\Sigma$.
\item If $(\theta,t)$ is on an edge of $\Lambda$, then $\Sigma_\theta$ and $F_t$ are transverse except for a single center or saddle tangency. Moreover, the tangency type does not alter along an edge of $\Lambda$. In other words, every edge represents a center or saddle tangency.
\item For any number $t\in(a,b)$, the horizontal circle $C_t:=S^1\times\{t\}\subset A$ contains at most one vertex of $\Lambda$. Similarly, for any angle $\theta\in [0,2\pi]$, the vertical interval $I_\theta:=\{\theta\}\times (0,1)$ contains at most one vertex of $\Lambda$ (see Figure \ref{a}).
\item A vertex of $\Lambda$ is either a \emph{birth-and-death} vertex with valence 2 as in Figure \ref{b}, or a \emph{crossing} vertex with valence 4 as in Figure \ref{c}.
\item The edges of $\Lambda$ are not tangent to any horizontal circle $C_s$ or vertical interval $I_\theta$ (see Figure \ref{a}).
\end{enumerate}
\begin{figure}[h!]
	\centering
	\begin{subfigure}[b]{0.3\textwidth}
		\centering
		\def\svgwidth{0.79\linewidth}
		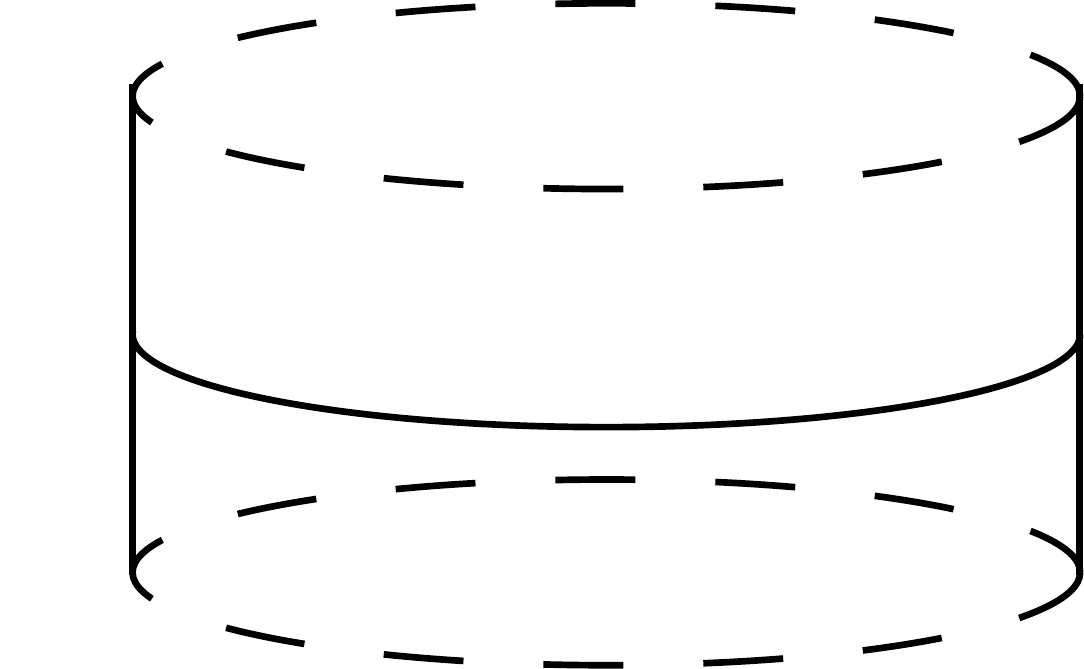
		\caption{Horizontal circles.\label{a}}
	\end{subfigure}%
	~ 
	\begin{subfigure}[b]{0.3\textwidth}
		\centering
		\def\svgwidth{0.7\linewidth}
		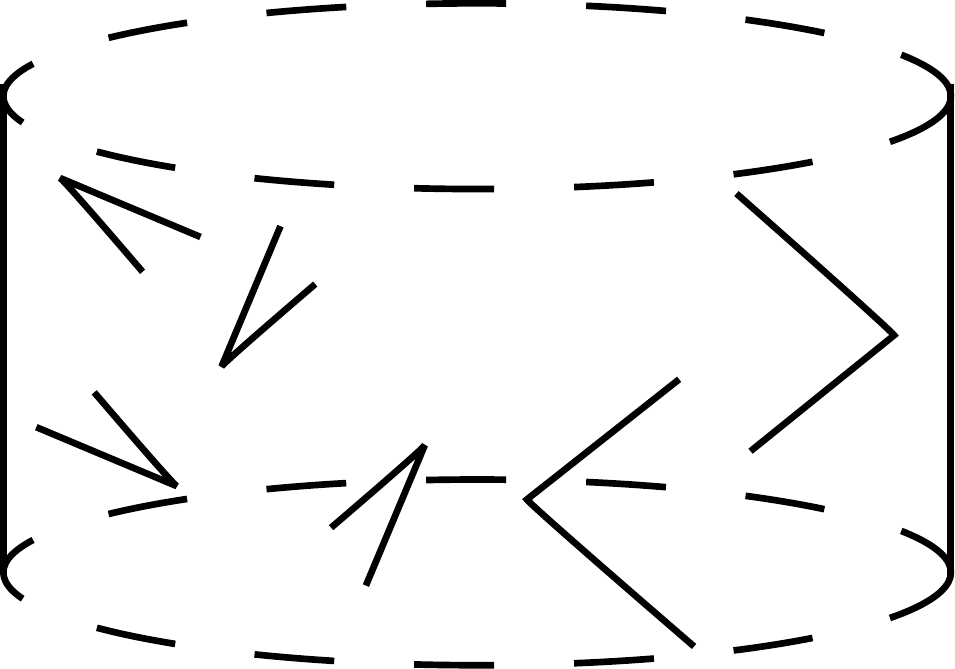
		\caption{Birth-and-death vertices.}\label{b}
	\end{subfigure}
		~ 
	\begin{subfigure}[b]{0.3\textwidth}
		\centering
		\def\svgwidth{0.7\linewidth}
		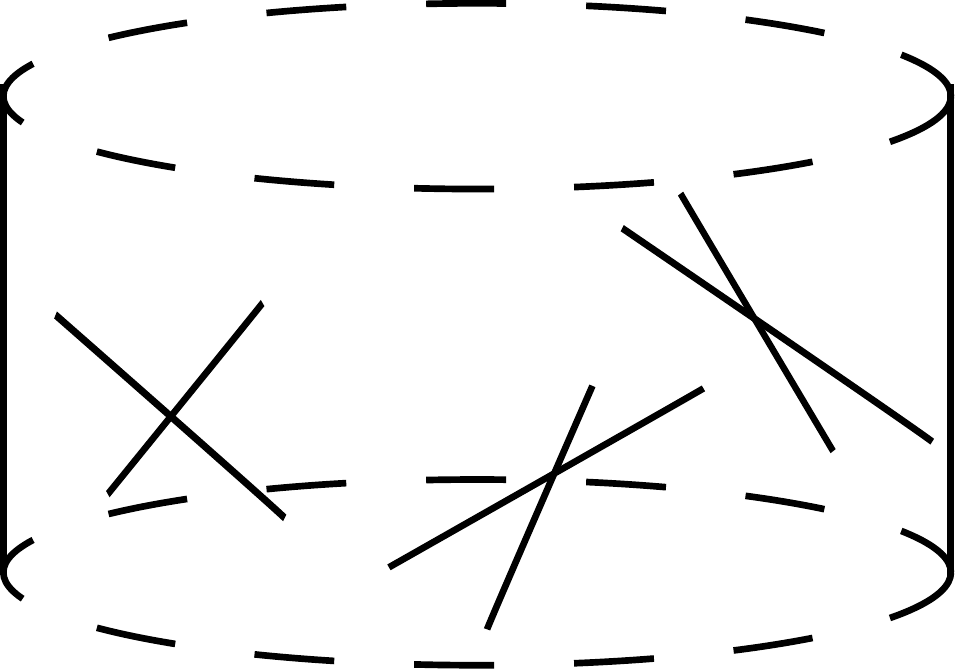
		\caption{Crossing vertices.\label{c}}
	\end{subfigure}
	\caption{Local pictures of the intersection graphic $\Lambda$.}
\end{figure}
\begin{Def}
The graph $\Lambda$ is called an \emph{intersection graphic} of the families $\Sigma_\theta$ and $F_t$. A connected component of the $A\setminus \Lambda$ is called a \emph{region} of $A\setminus \Lambda$.
\end{Def}

\noindent\textbf{Assumption.} For the rest of this section assume that the pages $\Sigma_{\theta}$ and the level surfaces $F_t$ of the middle slab are in a Cerf position providing an intersection graphic $\Lambda$ in the annulus $A$, satisfying the properties listed above.

\subsection{Labeling the levels of the middle slab}\label{labelling} Following Section 1 of \cite{GordonLuecke}, we label a level surface $F_t$ with $L$ (respectively with $H$) if there exists a page $\Sigma_\theta$ such that $\Sigma_\theta\cap F_t$ contains a properly embedded arc $\alpha\subset \Sigma_{\theta}$ that cuts off a half disk $\Delta^-$ (respectively $\Delta^+$) from $\Sigma_\theta$ such that the arc $\beta^-=\Delta^-\cap \partial \Sigma_\theta$ (respectively $\beta^+=\Delta^+\cap \partial \Sigma_\theta$) lies completely below (respectively above) $F_t$.  We say that $\Delta^-$/$\Delta^+$ is a \emph{low}/\emph{high} disk for $F_t$. Notice that, in the definition, the interior of $\Delta^\pm$ is allowed to include circles from the intersection of $\Sigma_\theta\cap F_t$.

One can also define the labelling for the regions of $A\setminus \Lambda$ in the following way: A region $R$ of $A\setminus \Lambda$ is labelled with $L$ (respectively with $H$) if there exists a point $(\theta,t)$ in $R$ such that an arc in $\Sigma_{\theta}\cap F_t$ cuts off a low disk (respectively a high disk) for $F_t$ from $\Sigma_\theta$.

\begin{Obs}
By properties of the Cerf position, if a region $R$ receives a label, then for every $t$, for which the horizontal circle $C_t$ meets $R$, the level $F_t$ receives the same label. Moreover, by definition of the labelling, a level $F_t$ is labelled with $L$ or $H$ if and only if there exists a region $R$ that receives the label $L$ or $H$, respectively, and meets the horizontal circle $C_t$.
\end{Obs}

\begin{Rem} We will see below that a level surface $F_t$ that intersects a page $\Sigma_\theta$ in arcs that are inessential in $\Sigma_{\theta}$ receives a label. Using thin position arguments, we will introduce a level surface $F_s$, which is not labelled, and therefore, has no inessential arcs of intersection with the pages. As in the proof of Lemma \ref{meridional}, such a surface will be an essential tool to execute a combinatorial argument which provides the complexity bound stated in Theorem \ref{ThinTheorem}.
\end{Rem}

Thin position arguments and the intersection properties of the pages $\Sigma_{\theta}$ and the levels $F_t$ will provide the following two lemmas which will be useful in detecting a surface $F_s$ that is not labelled.

\begin{Clm}\label{standard}For all $\delta>0$ sufficiently small, $F_{a+\delta}$ is labelled with $L$ and $F_{b-\delta}$ with $H$.
\end{Clm}
\begin{proof} Since the pages are standard  near $\partial X_K$ with respect to the level surfaces $F_t$, there exists a page $\Sigma_{\theta}$, which hangs down near the local maximum $b$. Therefore, for $t$ values sufficiently close to $b$, $\Sigma_{\theta}\cap F_t$ contains an arc that cuts off a high disk for $F_t$ from $\Sigma_{\theta}$, i.e., $F_t$ is labelled with $H$. Similarly, for $t$ values sufficiently close to $a$, the level surface $F_t$ is labelled with $L$.
\end{proof}

Since we essentially use  the same labelling as in Section 1 of \cite{GordonLuecke}, we immediately obtain the following.
\begin{Clm}\label{bothlabel}
There exists no $t\in (a,b)$ such that $F_t$ is labelled with both $L$ and $H$. In other words, every level surface $F_t$ receives at most one label. Hence, every region of $A\setminus \Lambda$ receives at most one label.
\end{Clm}
\begin{proof}
The first statement follows from  Lemma 1.1 in \cite{GordonLuecke}. It immediately follows that a region $R$ receives at most one label. Otherwise, if there is a region receiving both labels, then every level $F_t$, for which the horizontal circle $C_t$ meets $R$, receives both labels.
\end{proof}

The last two claims imply that the labels of $F_t$ change from $L$ to $H$ as $t$ increases from 0 to 1. Next we show that there must be a level surface $F_s$, which receives no label. We introduce this surface and analyze its properties in the following subsection.

\subsection{A special level.} Now we are ready to introduce a special level in the middle slab. Let 
$$s:=\textrm{sup}\{t\in (a,b)\,|\, F_t\text{ is labelled with }L\}.$$  The level surface $F_s$ will be the surface in the knot exterior which provides the complexity bound stated in Theorem \ref{ThinTheorem}. In this subsection, we will show that $F_s$ can be assumed to satisfy certain conditions towards the proof of the theorem, by proving the following.

\begin{Lem}\label{usefullevel}
For any angle $\theta$, any transversal arc in the intersection $\Sigma_\theta\cap F_{s}$ is essential in $\Sigma_ \theta$. Moreover, either Theorem \ref{ThinTheorem} holds 
or $F_s$ satisfies the following properties:
\begin{enumerate}
	\item The horizontal circle $C_s$ contains a crossing vertex of $\Lambda$.
	\item For any $\epsilon>0$, there exist numbers $s_-\in (s-\epsilon,s)$ and $s_+\in (s,s+\epsilon)$ such that $F_{s_-}$ is labelled with $L$ and $F_{s_+}$ is  labelled with $H$.
\end{enumerate}
\end{Lem}
We will prove Lemma \ref{usefullevel} at the end of this subsection. First, we present a discussion that provides a sequence of claims that are used in the proof of the lemma.
\begin{Clm}
The number $s$ equals neither $a$ nor $b$.
\end{Clm}
\begin{proof}
We immediately obtain $a<s$ from Claim \ref{standard} because for sufficiently small $\delta$ values, $F_{a+\delta}$ is labelled with $L$. On the other hand, assume for a contradiction that $s\ge b$, hence $s=b$. Then there exist $t$ values arbitrarily close to $b$ such that $F_t$ are labelled with $L$. By Claim \ref{standard}, such levels are labelled with $H$ as well, which is impossible by Claim \ref{bothlabel}.
\end{proof}

\begin{Clm}\label{nolabels}
The level surface $F_{s}$ is not labelled.
\end{Clm}
\begin{proof}
Assume that $F_{s}$ is labelled with either $L$ or $H$. We show that both cases lead to a contradiction. \vspace{2mm}

\noindent\textbf{Case 1.} $F_{s}$ is labelled with $L$: In this case, a transversal arc of intersection in $\Sigma_{\theta}\cap F_s$ that bounds a low disk for $F_{s}$ in $\Sigma_{\theta}$ persists in the intersection $\Sigma_\theta \cap F_t$ for any $t\in (s-\epsilon,s+\epsilon)$, for $\epsilon>0$ sufficiently small. Therefore, for every number $t\in (s,s+\epsilon)$, the level $F_{t}$ receives the label $L$. But this contradicts that $s$ is the supremum. \vspace{2mm}
	
\noindent\textbf{Case 2.} $F_{s}$ is labelled with $H$: In this case, a transversal arc of intersection in $\Sigma_{\theta}\cap F_s$ that bounds a high disk for $F_{s}$ in $\Sigma_{\theta}$ persists in $\Sigma_\theta \cap F_t$ for any $t\in (s-\epsilon,s+\epsilon)$, for $\epsilon>0$ sufficiently small. Therefore, for any $t\in (s-\epsilon, s)$, the level $F_t$ receives the label $H$. Since $F_s$ is not labelled with $L$ (by the previous case), there exists a number $t\in (s-\epsilon,s)$ such that $F_t$ receives the label $L$ as well. However, this contradicts Claim \ref{bothlabel}.
\end{proof}

\begin{Clm}\label{lbelow}
For every $\epsilon>0$, there exists $t\in(s-\epsilon,s)$ such that $F_t$ is labelled with $L$.
\end{Clm}
\begin{proof}
Since $s$ is the supremum of $L$-labelled levels and $F_s$ is not labelled (by Claim \ref{nolabels}), parameters of the $L$-labelled levels must be arbitrarily close to $s$.
\end{proof}

\begin{Clm}\label{ess}
For any angle $\theta$, transversal arcs in $\Sigma_\theta\cap F_{s}$ are essential in $\Sigma_ \theta$.
\end{Clm}
\begin{proof}
Assume for a contradiction that $\Sigma_\theta\cap F_{s}$ contains a transversal arc of intersection $\alpha$ that is inessential in $\Sigma_{\theta}$. We will show that $F_s$ is labelled, which contradicts  Claim \ref{nolabels}. 

\vspace{2mm}

\noindent\textbf{Case 1.} $\Sigma_{\theta}$ and $F_s$ intersect transversely: In this case, $\alpha$ cuts off a half disk $\Delta$ from $\Sigma_\theta$ such that $\Delta$ intersects $F_s$ transversely in embedded arcs and simple closed curves. Then an arc of intersection $\alpha'\subset \Delta\cap F_s\subset \Sigma_{\theta}\cap F_s$ that is outermost in $\Delta$ cuts off a half disk $\Delta'\subset \Delta$ which is a low or high disk for $F_s$ in $\Sigma_\theta$. This implies $F_{s}$ is labelled.

\vspace{2mm}

\noindent\textbf{Case 2.} $\Sigma_{\theta}$ and $F_s$ do not intersect transversely: In this case, $(\theta,s)$ is in the intersection graphic $\Lambda$ and we can find an angle $\theta'$ sufficiently close to $\theta$ so that
\begin{enumerate}
	\item[(i)] The point $(\theta',s)$ lies in a region of $A\setminus \Lambda$, that is, $\Sigma_{\theta'}$ is transverse to $F_s$, and
	\item[(ii)] The transversal intersection arc $\alpha$ persists in $\Sigma_{\theta'}\cap F_s$ as an inessential arc in $\Sigma_{\theta'}$.
\end{enumerate}
In other words, $\Sigma_{\theta'}$ and $F_s$ are as in the previous case. An identical argument yields a label for $F_s$.
\end{proof}

Notice that the last claim establishes the first statement in Lemma \ref{usefullevel}. Now we will introduce other claims of a different flavor to analyze the intersection graphic $\Lambda$.

\begin{Clm}\label{vertex}
If the horizontal circle $C_s\subset A$ contains no vertex of the intersection graphic $\Lambda$, then Theorem \ref{ThinTheorem} holds. 
\end{Clm}
\begin{proof}
If there exists no vertex of $\Lambda$ in $C_s$, then the level surface $F_{s}$ is in regular position with respect to pages. Moreover, by Claim \ref{ess}, every transversal arc of intersection in $\Sigma_\theta\cap F_s$ is essential in $\Sigma_\theta$ for any angle $\theta$. In other words, $F_{s}$ is a meridional perfect surface in $X_K$ (see Definition \ref{perfect}), where $g(F_{s})$ equals the Heegaard genus $g$ of $M$. By Lemma \ref{meridional}, we get $d_\A(\phi)\le3$ when $g=0$, and $d_\A(\phi)\le2g+2$ when $g\ge 1$. Thus, Theorem \ref{ThinTheorem} holds. 
\end{proof}

\begin{Clm}\label{habove}
If there exists an $\epsilon>0$ such that $F_t$ is not labelled for any $t\in(s,s+\epsilon)$, then Theorem \ref{ThinTheorem} holds.
\end{Clm}
\begin{proof}
Assume that there exists an $\epsilon>0$ such that for any $t\in (s,s+\epsilon)$, the level $F_{t}$ is not labelled. 
Then we can choose a number $s'\in (s,s+\epsilon)$ such that $F_{s'}$ receives no label and $C_{s'}$ contains no vertex of $\Lambda$. In other words, $F_{s'}$ satisfies the hypotheses of Claims \ref{ess} and \ref{vertex}. Applying identical arguments to $F_{s'}$, we deduce Theorem \ref{ThinTheorem} holds.
\end{proof}

\begin{Clm}\label{crossing}
If there exists a birth-and-death vertex on $C_s$, then Theorem \ref{ThinTheorem} holds.
\end{Clm}
\begin{proof}
Assume that $C_s$ contains a birth-and-death vertex. Recall that every horizontal circle in $A$ contains at most one vertex of $\Lambda$. So, away from the birth-and-death vertex, $C_s$ intersects edges of $\Lambda$ transversely. We introduce a case analysis depending on the location of the edges adjacent to the vertex on $C_s$, and we either reach a contradiction or show that Theorem \ref{ThinTheorem} holds.\vspace{2mm}

\noindent\textbf{Case 1.} One edge is above $C_s$, and the other is below: In this case, there exists an $\epsilon>0$ such that for any $t\in(s-\epsilon,s)$, the horizontal circle $C_t$ meets the same regions as $C_s$. Since $F_s$ is unlabelled by Claim \ref{nolabels}, all regions intersecting $C_s$ are unlabelled. In other words, for any $t\in(s-\epsilon,s)$, all regions intersecting $C_t$ are unlabelled. This implies that the level $F_t$ is unlabelled for any $t\in(s-\epsilon,s)$, which is impossible by Claim \ref{lbelow}. \vspace{2mm}

\noindent\textbf{Case 2.} Both edges are above $C_s$: In this case, again there exists an $\epsilon>0$ such that for any $t\in(s-\epsilon,s)$, the horizontal circle $C_t$ meets the same regions as $C_s$. Similarly, this implies that for any $t\in(s-\epsilon,s)$, $F_t$ is unlabelled, which is impossible by Claim \ref{lbelow}.\vspace{2mm}

\noindent\textbf{Case 3.} Both edges are below $C_s$: In this case, there exists an $\epsilon>0$ such that for any $t\in(s,s+\epsilon)$, the horizontal circles $C_t$ and $C_s$ meet the same regions. Similarly, $F_t$ is unlabelled for $t\in (s,s+\epsilon)$. Thus, by Claim \ref{habove}, Theorem \ref{ThinTheorem} holds.
\end{proof}

Now we are ready to prove Lemma \ref{usefullevel} and finish this subsection.\vspace{2mm}

\noindent\textit{Proof of Lemma \ref{usefullevel}.} By Claim \ref{ess}, any transversal arc of intersection in $\Sigma_\theta\cap F_s$ is essential in $\Sigma_{\theta}$.

Now assume that (1) does not hold. Then either (a) $C_s$ contains no vertex or (b) $C_s$ contains a birth-and-death vertex. In case (a), Theorem \ref{ThinTheorem} holds by Claim \ref{vertex}. In case (b), Theorem \ref{ThinTheorem} holds by Claim \ref{crossing}.

Finally, assume (2) does not hold. By Claim \ref{lbelow}, for any $\epsilon>0$, there exists $s_-\in (s-\epsilon,s)$ such that $F_{s_-}$ is labelled with $L$. Therefore, if (2) does not hold, then there exists an $\epsilon>0$ such that for any $t\in(s,s+\epsilon)$, $F_{t}$ is not labelled. Thus, by Claim \ref{habove}, Theorem \ref{ThinTheorem} holds.\qed

\subsection{Analyzing the crossing vertex.} In the previous subsection we showed that for $s=\sup\,\{t\in (a,b)\,|\, F_t \text{ is labelled with } L\}$, any transversal intersection arc in $\Sigma_\theta \cap F_s$ is essential in $\Sigma_{\theta}$. Moreover, in Lemma \ref{usefullevel}, we showed that if $F_s$ does not satisfy one of the following properties, then Theorem \ref{ThinTheorem} holds:
\begin{enumerate}
	\item The horizontal circle $C_s$ contains a crossing vertex of $\Lambda$.
	\item For any $\epsilon>0$, there exist numbers $s_-\in (s-\epsilon,s)$ and $s_+\in (s,s+\epsilon)$ such that $F_{s_-}$ is labelled with $L$ and $F_{s_+}$ is  labelled with $H$.
\end{enumerate}
Since, our ultimate goal is to prove Theorem \ref{ThinTheorem}, in this subsection, we assume that $F_s$ satisfies (1) and (2), and we analyze $F_s$ further to prove some claims that will be used in the proof of the theorem.

Let $(\psi,s)$ be the crossing vertex of $\Lambda$ that is in $C_s$. By rotating the open book, if necessary, we can assume that $\psi$ is a non-zero angle, and $\Sigma_0$ is transverse to $F_{s}$. Let $R^+$ be the region that is adjacent to the edges above $C_s$ at $(\psi,s)$, $R^-$ the region that is adjacent to the edges below $C_s$ at $(\psi,s)$. Moreover, let $R^w$ (respectively $R^e$) be the region to the west (respectively to the east) of the vertex $(\psi,s)$. Let the four edges adjacent to the vertex $(\psi,s)$ be $e_1, e_2,e_3,e_4$, as in Figure \ref{psis}.
\begin{figure}[h]
	\centering
	\def\svgwidth{0.5\linewidth}
	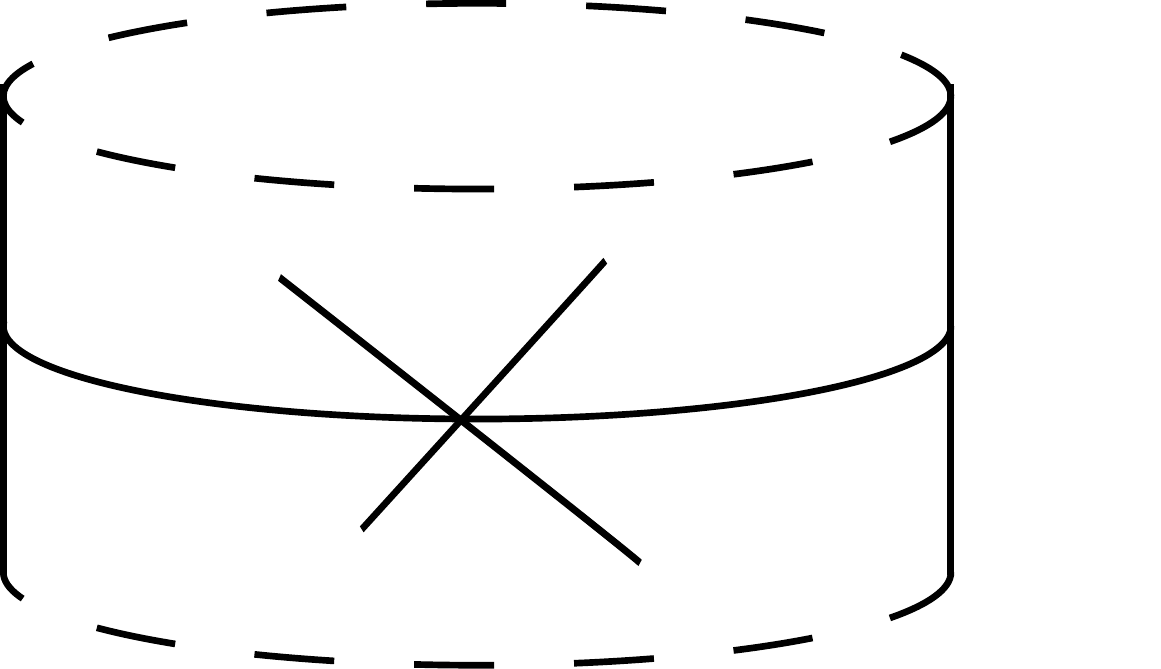
	\caption{The local picture of $\Lambda$ near the crossing vertex $(\psi,s)$.}\label{psis}
\end{figure}

\begin{Clm}\label{entangled}
The region $R^+$ is labelled with $H$ and $R^-$ is labelled with $L$. 
The surfaces $\Sigma_{\psi}$ and $F_{s}$ intersect transversely except for two saddle tangencies. Moreover, the two saddle tangencies are entangled, i.e., $\Sigma_\psi\cap F_s$ has a connected singular component containing both saddle tangencies.
\end{Clm}
\begin{proof}
Assume for a contradiction that $R^-$ is not labelled. Let $R_1,\ldots, R_n$ be the regions that meet $C_s$. By properties of the intersection graphic $\Lambda$, the horizontal circle $C_s$ intersects edges of $\Lambda$ transversely away from the vertex $(\psi,s)$. Then there exists an $\epsilon>0$ such that for any $t\in (s-\epsilon,s)$, $C_t$ intersects the regions $R_1,\ldots,R_n$, and $R^-$. Since $C_s$ is not labelled, none of the regions $R_i$ are labelled. Since $R^-$ is not labelled either, it follows that $C_{t}$ meets no labelled regions. This implies that $F_t$ is not labelled for $t\in(s-\epsilon,s)$, which contradicts the assumption (2) above. On the other hand, if we assume that $R^+$ is not labelled, it follows from the same argument that there exists an $\epsilon>0$ such that $F_t$ is not labelled for $t\in(s,s+\epsilon)$, which again contradicts the assumption (2) above. 


To prove the second claim, choose $\epsilon>0$ sufficiently small so that $(\psi-\epsilon,s)\in R^w$ and $(\psi,s+\epsilon)\in R^+$. If we travel from $(\psi-\epsilon,s)$ to $(\psi,s+\epsilon)$ along the straight line between them, we cross $\Lambda$ once at the edge $e_1$. Since $R^w$ is not labelled and $R^+$ is labelled, this implies that the tangency represented by $e_1$ changes arc types in the intersections. Thus, $e_1$ must represent a saddle tangency rather than a center tangency. A similar argument implies that $e_2$ must represent a saddle tangency as well, because $R^+$ is labelled and $R^e$ is not labelled. Thus, the edges $e_1$ and $e_2$ represent two saddle tangencies between $\Sigma_\psi$ and $F_s$. Finally, observe that  the saddle tangency represented by $e_1$ introduces inessential arcs of intersection and the saddle tangency represented by $e_2$ eliminates the same inessential arcs of intersection. Thus, the two saddle tangencies meet the same singular component of $\Sigma_\psi\cap F_s$, which implies that they are entangled.
\end{proof}

\begin{Clm}\label{notclosed}
Let $G$ be the singular component of $\Sigma_\psi \cap F_s$ containing the two entangled saddles. Then $G$ meets $\partial \Sigma_\psi$.
\end{Clm}
\begin{proof} Assume for a contradiction that $\partial \Sigma_\psi\cap G=\emptyset$. This implies that no arcs in $\Sigma_{\psi-\epsilon} \cap F_{s}$ interact with the entangled saddles. In particular, if we travel from $(\psi-\epsilon,s)$ to $(\psi,s+\epsilon)$ across the edge $e_1$ between them, the entangled saddle represented by $e_1$ does not alter the arc types in $\Sigma_{\psi-\epsilon} \cap F_{s}$. Thus, every arc in $\Sigma_{\psi} \cap F_{s+\epsilon}$ is essential in $\Sigma_\psi$, and so $R^+$ is not labelled, which is impossible by Claim \ref{entangled}.
\end{proof}

Now we will analyze how the entangled saddles of $F_s$ to the page $\Sigma_\psi$ affect the type of intersection arcs from $\Sigma_{\psi-\epsilon}\cap F_s$ to $\Sigma_{\psi+\epsilon}\cap F_s$.  
Fix $\epsilon>0$ small enough so that $\Sigma_{\psi}$ is the only critical page in $\Sigma\times[ \psi-\epsilon,\psi+\epsilon]$, and let $\wh F$ be the component of $(\Sigma\times[ \psi-\epsilon,\psi+\epsilon])\cap F_s$ that contains the singular component of $\Sigma_\psi\cap F_s$.


\begin{Clm}\label{distance} For any pair of arcs $\alpha^\pm\subset \Sigma_{\psi\pm \epsilon}\cap \wh F$, we have $d_\A(\alpha^+,\alpha^-)\le 2$.
\end{Clm}

\begin{proof}
Let $N(G)$ be a neighborhood of the singular component $G$ of $\Sigma_{\psi}\cap F_s$ in $\Sigma_\psi$. Notice that $G\subset \Sigma_\psi$ is a graph with two vertices of valence 4 away from $\partial \Sigma_{\psi}$, where we assume $g(\Sigma_\psi)\ge 2$ as stated in Theorem \ref{ThinTheorem}. Therefore, $N(G)$ does not fill the surface $\Sigma_\psi$, i.e., there exists an essential arc, say $\beta$, in $\Sigma_\psi$ disjoint from $N(G)$.

Let $q:\Sigma\times [\psi-\epsilon,\psi+\epsilon]\to \Sigma_\psi$ be the projection map. It follows that $q(\alpha^\pm)\subset N(G)$ up to isotopy. Therefore, the arc $\beta$ is disjoint from $q(\alpha^+)$ and $q(\alpha^-)$ up to isotopy. Thus, we get $d_\A(\alpha^+,\alpha^-)\le 2$.
\end{proof}

\begin{Clm}\label{int}
Any simple closed curve in $\Sigma_{\psi\pm\epsilon} \cap\wh F$ is non-trivial $\Sigma_{\psi\pm\epsilon}$.
\end{Clm}
\begin{proof}  Without loss of generality, assume for a contradiction that $\Sigma_{\psi-\epsilon} \cap\wh F$ contains a simple closed curve $\alpha$ that is trivial in $\Sigma_{\psi-\epsilon}$. We can also assume that this curve interacts with the saddle tangency represented by the edge $e_1\subset \Lambda$ (see Figure \ref{psis}). So, if we travel from $(\psi-\epsilon,s)$ to $(\psi,s+\epsilon)$ along the straight line between them, an essential arc enters into the saddle tangency with the trivial curve $\alpha$, and the arc types in the intersection do not change. This, in particular, implies that every arc in $\Sigma_{\psi} \cap F_{s+\epsilon}$ is essential in $\Sigma_\psi$, and so $R^+$ is not labelled, which is impossible by Claim \ref{entangled}.
\end{proof}

\begin{Lem}\label{tangencies}
The number of essential saddles of $\Sigma\times ([0,\psi-\epsilon]\cup [\psi+\epsilon, 2\pi])\cap F_s$ is at most $-\chi(F_s)-2$.
\end{Lem}
\begin{proof}
Fix an $\epsilon >0$ such that there are exactly two saddle tangencies of the level surface $F_s$ in $\Sigma\times [\psi-\epsilon,\psi+\epsilon]$. Let $c$ be the number of center tangencies and $s_i$ (respectively $s_e$) the number of inessential (respectively essential) saddle tangencies of $F_s$ to the pages in $\Sigma\times ([0,\psi-\epsilon]\cup [\psi+\epsilon, 2\pi])$. A standard Euler characteristic calculation provides 
$$\chi(F_s)+2=c-s_i-s_e \implies s_e = -\chi(F_s)-2+(c-s_i).$$
So, it suffices to show that $c-s_i\le 0$, or equivalently, $c\le s_i$.

By the last two claims, the saddle tangencies of $\Sigma_\psi\cap F_s$ are neither contained in a subdisk of $F_s$ nor do they interact with any inessential curve of intersection in $\Sigma_{\psi\pm\epsilon} \cap F_s$. Hence, it follows from the arguments of Lemma \ref{saddles} that away from the entangled saddles we have $c\le s_i$, as desired.
\end{proof}
\subsection{Proof of Theorem \ref{ThinTheorem}.}
As in the proof of Lemma \ref{meridional}, the result essentially follows from a counting argument that measures how much the arc types change as we travel from $\Sigma_0$ to $\Sigma_{2\pi}$ along a level surface in $X_K$ through the saddle tangencies. The counting arguments slightly differ between the cases $g=0$ and $g\ge 1$. Therefore, at the end, we will provide different proofs for the three-sphere and positive genus three-manifolds. However, first let us provide the arguments that are common to both cases.

Consider the meridional surface $F_s$, where $s:=\sup\{t\in (a,b)\,|\, F_t \text{ is labelled with } L\}$ with the labelling defined in Subsection \ref{labelling} above. By Lemma \ref{usefullevel}, any transversal arc of intersection in $\Sigma_\theta\cap F_s$ is essential in $\Sigma_{\theta}$. We can assume that $F_s$ is transverse to the page $\Sigma_0=\Sigma_{2\pi}$ by slightly rotating the pages, if necessary. Moreover, by Claim \ref{entangled}, we can assume that there exists an angle $\psi\neq 0$ such that the page $\Sigma_\psi$ is transverse to $F_s$ except for two entangled saddle tangencies. For any angle $\theta\neq \psi$, the level $F_s$ is transverse to $\Sigma_\theta$ except for possibly a single center or saddle tangency. 

Let there be $2n$ boundary components of $F_s$. 
We can denote $\partial F_s$ as $\{x_1,\ldots,x_{2n}\}\times S^1\subset \partial \Sigma\times S^1\cong \partial X_K$, where $x_i$ are distinct points in $\partial \Sigma$. By Claim \ref{notclosed}, the singular component $G$ of $\,\Sigma_{\psi}\cap F_{s}$ meets $\partial \Sigma_{\psi}\subset\partial X_K$. It follows that $G$ meets $ \partial X_K$ at either 2, 4, or 6 points. For simplicity, let us say an endpoint $x_i$ is \emph{singular} if $\{x_i\}\times S^1$ meets the singular component $G$. Otherwise, say $x_i$ is \emph{non-singular}. Hence, among $x_1,\ldots,x_{2n}$, there are either 2, 4, or 6 singular endpoints. We denote the number of singular endpoints by $r$.

Let $S$ be the preimage of $F_s$ under the quotient map $q: \Sigma\times [0,2\pi]\to X_K$, which maps $\Sigma\times\{\theta\}$ to $\Sigma_\theta$ in the natural way. For simplicity, we will not distinguish $\Sigma\times\{\theta\}$ and $\Sigma_\theta$. Observe that $S\cap (\partial \Sigma\times [0,2\pi])$ consists of vertical arcs $\{x_1,\ldots,x_{2n}\}\times [0,2\pi]$. Moreover, since $\Sigma_0$ is transverse to $F_s$, the intersection $S\cap \Sigma_0$ consists of some simple closed curves and exactly $n$ essential arcs, which are bounded by $\{x_1,\ldots,x_{2n}\}$, so that $S\cap \Sigma_{2\pi}$ consists of images of those curves and arcs under the monodromy, $\phi$.

By Lemma \ref{tangencies}, there are at most $-\chi(F_s)-2=(2g+2n-4)$ essential saddle tangencies of $F_s$ to the pages in $\Sigma\times ([0,\psi-\epsilon]\cup [\psi+\epsilon, 2\pi])$ for $\epsilon>0$ sufficiently small. As $\theta$ increases from $0$ to $2\pi$, the arc types in $S\cap \Sigma_\theta$ can change only if a page contains an essential saddle of $F_s$. Moreover, as we pass through each essential saddle tangency away from $\Sigma_{\psi}$, at most two new arcs can be introduced. As we pass through the entangled saddles in $\Sigma_{\psi}$, at most $r/2$ arc types are introduced. Therefore, the maximum possible total number of essential arc types that are introduced by essential or entangled saddle tangencies is $2(2g+2n-4)+r/2=4g+4n-8+r/2$. With the $n$ arcs in $S\cap \Sigma_0$, we deduce that the preimage $S=q^{-1}(F_s)$ intersects the pages of $\,\Sigma\times [0,2\pi]$ in at most $4g+5n-8+r/2$ distinct essential arc types. For $i=1,\ldots,2n$, let $k_i$ be the number of essential arcs that have endpoints in $x_i$. Since each arc has two endpoints, when we add $k_i$'s, we get
\begin{align}\label{ineq1}
k_1+k_2+\ldots+k_{2n}=8g+10n-16+r\le 8g+10n-10,
\end{align}
as the number $r$ of singular endpoints is at most 6.
Inequality \ref{ineq1} will allow us to apply combinatorial arguments. Now let us prove the main theorem of this section.\vspace{2mm}

\noindent\textit{Proof of Theorem \ref{ThinTheorem} for the three-sphere.} For $M=S^3$, the Heegaard genus is $g=0$ and by Inequality \ref{ineq1} we obtain 
$$k_1+k_2+\ldots+k_{2n}\le10n-10.$$

\noindent\textbf{Case 1.} $n=1$: In this case, the level surface $F_s$ is an annulus. Therefore, $F_s$ does not have any essential saddle tangencies to the pages $\Sigma_\theta$ and it is a perfect meridional annulus in $X_K$. It follows from Lemma \ref{meridional} that $d_\A(\phi)=0\le3$.\vspace{2mm}

\noindent\textbf{Case 2.} $2\le n\le 4$: In this case, there is an endpoint in $\partial \Sigma$ that realizes at most 3 distinct arc types. Otherwise, if every endpoint realizes at least four arc types, i.e., if $k_i\ge 4$ for all $i=1,\ldots,2n$, then from Inequality \ref{ineq1} we obtain
$$8n\le k_1+k_2+\ldots+k_{2n}\le 10n-10,$$
which is false for $2\le n\le 4$. Now let $x_1$ be the endpoint that realizes at most 3 distinct arc types, say $\alpha_1,\alpha_2,\alpha_3$. Since $\alpha_1\subset \Sigma_0$ and $\alpha_3\subset \Sigma_{2\pi}$ have the same endpoint $x_1$, we deduce that $\phi(\alpha_1)=\alpha_3$.

If $x_1$ is a non-singular endpoint, then $d_\A(\alpha_1,\alpha_2)\le 1$ and $d_\A(\alpha_2,\alpha_3)\le 1$ since the arc types are introduced by essential saddle tangencies away from $\Sigma_\psi$. Hence, we obtain $$d_\A(\phi)\le d_\A(\alpha_1,\phi(\alpha_1))=d_\A(\alpha_1,\alpha_3)\le 2\le 3.$$

On the other hand, if $x_1$ is a singular endpoint, then assume without loss of generality that $\alpha_2$ is introduced as $\alpha_1$ interacts with the entangled saddles at $\Sigma_\psi$. It follows from Claim \ref{distance} that $d_\A(\alpha_1,\alpha_2)\le 2$. Finally, we have $d_\A(\alpha_2,\alpha_3)\le 1$, which provides $$d_\A(\phi)\le d_\A(\alpha_1,\phi(\alpha_1))=d_\A(\alpha_1,\alpha_3)\le 3.$$

\noindent\textbf{Case 3.} $n\ge 5$: In this case,  Inequality \ref{ineq1}, $k_1+k_2+\ldots+k_{2n}\le10n-10$, implies that either there is an endpoint realizing 3 distinct arc types, or there are at least 10 endpoints realizing 4 distinct arc types. If there is an endpoint  realizing 3 distinct arc types, then the argument of the previous case implies $d_\A(\phi)\le 3$. So, assume that there are at least 10 endpoints realizing 4 distinct arc types. In particular, there exists a non-singular endpoint, say $x_1$, realizing $4$ distinct arc types. Let $\alpha_1,\alpha_2,\alpha_3,\alpha_4$ be the arc types that are realized by $x_1$. It follows that $d(\alpha_j,\alpha_{j+1})\le 1$, for $j=1,2,3$, since no $\alpha_j$ is involved with the entangled saddles. Since $\phi(\alpha_1)=\alpha_4$, we obtain 
$$d_\A(\phi)\le d_\A(\alpha_1,\phi(\alpha_1))=d_\A(\alpha_1,\alpha_4)\le 3.$$ 
This completes the proof of Theorem \ref{ThinTheorem} for the three-sphere. \qed\vspace{2mm}

Now le us present a proof for three-manifolds with positive Heegaard genus.\vspace{2mm}

\noindent\textit{Proof of Theorem \ref{ThinTheorem} for $g\ge 1$.} For positive genus three-manifolds, we analyze three different cases.\vspace{2mm}

\noindent\textbf{Case 1.} $n=1$: In this case, we do not use Inequality \ref{ineq1} since it provides an upper bound greater than the desired one. Instead, notice that the level surface $F_s$ we previously fixed is a genus $g$ surface with two boundary companents, i.e., $\chi(F_s)=-2g$. Then, by Lemma \ref{tangencies}, there are at most $2g-2$ saddle tangencies away from the entangled saddle. Moreover, each regular page $\Sigma_{\theta}$ contains a single arc of intersection from $\Sigma_\theta \cap F_s$, which is bounded by the $2n=2$ endpoints, say $\{x_1,x_2\}$. It follows that $x_1$ realizes at most $2g-1$ arc types between essential and entangled saddles, say $\alpha_1,\alpha_2,\ldots,\alpha_{2g-1}$. Since $\alpha_1\subset \Sigma_0$ and $\alpha_{2g-1}\subset \Sigma_{2\pi}$ have the same endpoint $x_1$, we deduce that $\phi(\alpha_1)=\alpha_{2g-1}$. Without loss of generality, assume that $\alpha_2$ is introduced as $\alpha_1$ interacts with the entangled saddles at $\Sigma_\psi$. It follows from Claim \ref{distance} that $d_\A(\alpha_1,\alpha_2)\le 2$. On the other hand, we have $d_\A(\alpha_j,\alpha_{j+1})\le 1$ for $j=2,\ldots, 2g-1$, which implies
$$d_\A(\phi)\le d_\A(\alpha_1,\phi(\alpha_1))=d_\A(\alpha_1,\alpha_{2g-1})\le 2g< 2g+2.$$

\noindent\textbf{Case 2.} $n=2$: In this case, Inequality \ref{ineq1}, $k_1+k_2+\ldots+k_{2n}\le 8g+10n-10$, turns into $k_1+k_2+k_3+k_4\le8g+10$. Therefore, there exists an endpoint, say $x_1$, realizing at most $2g+2$ arc types, say $\alpha_1,\alpha_2,\ldots,\alpha_{2g+2}$. Since $\alpha_1\subset \Sigma_0$ and $\alpha_{2g+2}\subset \Sigma_{2\pi}$ have the same endpoint $x_1$, we deduce that $\phi(\alpha_1)=\alpha_{2g+2}$.

If $x_1$ is not a singular endpoint, then $d_\A(\alpha_j,\alpha_{j+1})\le 1$ for each $j=1,\ldots,2g+1$ since the arc types are introduced by essential saddle tangencies away from $\Sigma_\psi$. Hence, we obtain $$d_\A(\phi)\le d_\A(\alpha_1,\phi(\alpha_1))=d_\A(\alpha_1,\alpha_{2g+2})\le 2g+1\le 2g+2.$$

On the other hand, if $x_1$ is a singular endpoint, then assume without loss of generality that $\alpha_2$ is introduced as $\alpha_1$ interacts with the entangled saddles at $\Sigma_\psi$. It follows from Claim \ref{distance} that $d_\A(\alpha_1,\alpha_2)\le 2$. On the other hand, we have $d_\A(\alpha_j,\alpha_{j+1})\le 1$ for $j=2,\ldots, 2g+1$, which provides $$d_\A(\phi)\le d_\A(\alpha_1,\phi(\alpha_1))=d_\A(\alpha_1,\alpha_{2g+2})\le 2g+2.$$

\noindent\textbf{Case 3.} $n\ge 3$: In this case, Inequality \ref{ineq1}, $k_1+k_2+\ldots+k_{2n}\le 8g+10n-10$,
implies that there is an endpoint realizing at most $\lfloor 5+(8g-10)/6 \rfloor\le 2g+2$ distinct arc types (which can be seen by a case analysis on values of $g$). The argument of the previous case works equally in this case. Thus, we get $d_\A(\phi)\le 2g+2$, as desired. \qed

\section{Non-Primitive Fibered Knots on Strongly Irreducible Heegaard Surfaces} \label{Non-Primitive}
In the last two sections, we showed that the complexity bound stated in Part (4) of Theorem \ref{MainTheorem} holds when a minimal genus Heegaard splitting $P\subset M$ is weakly reducible, or $K$ cannot be isotoped into $P$ in $M$. The remaining case is that the fibered knot $K$ lies on a strongly irreducible Heegaard surface $P$, which is addressed in the following theorem.\v

\begin{Thm}\label{KinP} Let $K\subset M$ be a fibered knot with monodromy $\phi$ and pages of genus greater than one. If $(P,U,V)$ is a strongly irreducible Heegaard splitting of genus $g\ge 2$ in $M$ such that $K\subset P$, then at least one of the following holds:
	\begin{enumerate}
		\item $P$ is isotopic to the Heegaard surface induced by $K$.
		\item $d_{\C}(\phi)\le 2g-2$.
\end{enumerate}
\end{Thm}
Notice that we state the theorem for any strongly irreducible Heegaard splittings rather than minimal genus ones, to prove Theorem \ref{StabilizationTheorem} as well. This section and the next are devoted to the proof of Theorem \ref{KinP}. In this section, we will give a proof of the theorem when $K\subset P$ is not primitive in either $U$ or $V$. 

\subsection{Curves on Heegaard splittings.} In this subsection, we state some fundamental definitions and facts regarding the types of simple closed curves in Heegaard surfaces.
\begin{Def}
Let $P\subset M$ be a Heegaard surface that bounds a handlebody $U$ on one side. A simple closed curve $C\subset P$ is called
\begin{itemize}
\item \emph{primitive} in $U$ if there exists an essential disk $D$ in $U$ such that $C$ intersects the boundary of $D$ exactly once,
\item  \emph{disk-busting} in $U$ if $C$ intersects every essential disk in $U$, and
\item a \emph{core} in $U$ if $C$ is isotopic into a spine of the handlebody $U$.
\end{itemize}
\end{Def}

Let $(P,U,V)$ be a Heegaard splitting and $C\subset P$ a simple closed curve. It follows from standard arguments in the Heegaard splittings literature that $C\subset P$ is primitive in $U$ if and only if it is a core in $U$ (see Lemma 4.45 in \cite{JohnsonNotes}). Moreover, if $C$ is a core in $U$, then $P$ cuts the exterior of $C$, $X_C\subset M$, in to a pair $(U',V)$, where $U'=U\setminus \mathring{N}(C)$ is a compression-body and $V$ is a handlebody. In other words, $P$ is a Heegaard surface in $X_C$. For more details, see \cite{JohnsonNotes}.

One convenience of working with a strongly irreducible Heegaard surface $P$ is the following lemma, which will allow us to manipulate the compressions of $P$.
\begin{Lem}[Scharlemann's no-nesting Lemma, \cite{Scharlemann}]\label{nonest} Let $(P,U,V)$ be a strongly irreducible Heegaard surface in a compact three-manifold $M$. If $\alpha\subset P$ is a simple closed curve that bounds a disk in $M$, then $\alpha$ bounds a properly embedded disk in $U$ or $V$.
\end{Lem}
Notice that if the curve $\alpha$ is essential in $P$, then by definition, it bounds an essential disk in $U$ or $V$. The following lemma will also be useful to analyze compressions of strongly irreducible Heegaard splittings.

\begin{Lem}\label{compressions}
Let $(P,U,V)$ be a strongly irreducible Heegaard splitting in $M$. If $P'$ is obtained by compressing $P$ along essential disks in $U$ (or $V$), then $P'$ is not a Heegaard surface.
\end{Lem}
\begin{proof}
Let $P'$ be a closed connected surface that is obtained by compressions of $P$ along essential disks $D_1,\ldots, D_k$ in $U$. First notice that if $P'$ is a sphere, then it cannot be a Heegaard splitting in $M$, for otherwise $M$ would be homeomorphic to $S^3$, which has no strongly irreducible Heegaard splittings by \cite{Wald}. So, we can assume that $P'$ has positive genus. The compressions of $P$ in $U$ can be interpreted as attaching two-handles $N(D_1),\ldots,N(D_k)$ to $V$, and so $P'$ bounds the subsapce $V' =V\cup (\cup_{i=1}^k N(D_i))$  in $M$. To prove that $P'$ is not a Heegaard splitting, we will show that $V'$ is not a handlebody.

If we isotope $P$ slightly into $V$, then it bounds the handlebody $V\setminus (P\times(0,1])$ on one side and the compression-body $(P\times[0,1]) \cup (\text{two-handles})$ on the other side. In other words, $P$ is a Heegaard splitting in $V'$ up to isotopy.  Since $P$ is a strongly irreducible Heegaard splitting in $M$, it follows that $P$ is strongly irreducible in the subspace $V'$ as well. Then we deduce from Theorem 2.1 in \cite{CassonGordon} that $V'$ is irreducible, i.e., it does not contain essential disks or spheres. Therefore, $V'$ is not a handlebody, and thus $P'$ is not a Heegaard splitting.
\end{proof}

\subsection{Non-meridional essential surfaces in fibered knot exteriors.} In the proof Theorem \ref{KinP} under the assumption that $K\subset P$ is not primitive, the surface $P\setminus\mathring N(K)$ embedded in $X_K$ will play a fundamental role. Since $K$ is assumed to lie in $P$, this surface will have two boundary components which realize a non-meridional  (possibly zero) slope in $\partial X_K$. Therefore, first we will provide some complexity bounds when $X_K$ contains a non-meridional essential surface.

\begin{Lem}\label{nonzeroslope}
Let $F\subset X_K$ be a properly embedded essential surface with non-empty boundary components of a non-meridional, non-zero slope in $\partial X_K$. Assume that $F$ is not boundary-parallel in $X_K$. If $F$ is an annulus, then $d_\A(\phi)\le 1$. If $\chi(F)\le -1$, then $d_\A(\phi)\le -\chi(F)$.
\end{Lem}
\begin{proof}
By Theorem 4 in \cite{Thurs}, we can isotope $F$ in $X_K$ so that $F$ only has saddle tangencies to $m=-\chi(F)$ pages. Moreover, since $F$ is an essential surface that is not boundary-parallel, every arc of intersection $F\cap\Sigma_{\theta}$ is essential in $\Sigma_{\theta}$ for any $\theta$ (see Lemma \ref{boundaryperfect}). Let $S$ be the preimage of $F$ under the quotient map $q: \Sigma\times [0,2\pi]\to X_K$, which maps $\Sigma\times\{\theta\}$ to $\Sigma_\theta$ in the natural way. For simplicity, we will not distinguish $\Sigma\times\{\theta\}$ and $\Sigma_\theta$.

If $F$ is an annulus, then $m=-\chi(F)=0$, i.e., $F$ has no tangencies to the pages. Fix an arc $\alpha\subset S\cap \Sigma_0$. Since there are no tangencies, there exists an arc $\beta\subset S\cap\Sigma_{2\pi}$, which is isotopic to $\alpha$ (when both are regarded in $\Sigma$). Since $F$ is properly embedded in $X_K$, either $\beta=\phi(\alpha)$, or $\beta$ and $\phi(\alpha)$ are disjoint. Thus, we obtain $d_\A(\alpha,\phi(\alpha))=d_\A(\beta,\phi(\alpha))\le 1$, which implies that $d_\A(\phi)\le 1$.

Now assume that $\chi(F)\le -1$, so there exist $m=-\chi(F)\ge 1$ saddle tangencies of $F$ to the pages. Let $\Sigma_{\theta_1}, \ldots, \Sigma_{\theta_m}$ be the pages that are transversal to $F$ except for a single saddle tangency, where $0< \theta_1< \ldots < \theta_m<2\pi$. For each $i=1,\ldots,m-1$, fix an angle $t_i$ in $(\theta_i,\theta_{i+1})$ and choose an arc $\alpha_i\subset S\cap \Sigma_{t_i}$. Furthermore, choose an arc $\alpha_0 \subset S\cap \Sigma_0$ and set $\alpha_m=\phi(\alpha_0)\subset S\cap \Sigma_{2\pi}$. Since, for each $i=0,\ldots,m-1$, there is only a single saddle tangency of $F$ in $\Sigma\times [t_i,t_{i+1}]$, we can isotope $\alpha_{i+1}$ and $\alpha_i$ to be disjoint in $\Sigma$. In other words, for each $i=0,\ldots,m-1$, we have $d_\A(\alpha_i,\alpha_{i+1})\le 1$. Thus, by the triangle inequality, we obtain
$$d_\A(\phi)\le d_\A(\alpha_0,\phi(\alpha_0))=d_\A(\alpha_0,\alpha_m)\le \sum_{i=0}^{m-1}d_\A(\alpha_i,\alpha_{i+1})\le m=-\chi(F),$$
as desired.
\end{proof}
\begin{Rem}
We believe that the complexity bound stated in the last lemma can be given in terms of the genus rather than the Euler characteristic of $F$, by a careful application of the combinatorial arguments introduced in the proof of Lemma \ref{meridional}. This would be more convenient especially when the number of boundary components of $F$ is large. However, in the proof of Theorem \ref{KinP} for non-primitive knots, we will be dealing with surfaces that have small number of boundary components. Therefore, a complexity bound in terms of Euler characteristic is fine for our purposes.
\end{Rem}
Next, we state three lemmas that will be useful when we have an incompressible surface in $X_K$ with boundary components of the zero slope. We begin with the following lemma, which essentially follows from Proposition 3.1 in \cite{FW}.

\begin{Lem}\label{standa}
Let $F\subset X_K$ be a properly embedded incompressible surface. If $F$ is disjoint from a page $\Sigma_\theta$, then each component of $F$ is either a $\partial$-parallel annulus or isotopic to a page in $X_K$.
\end{Lem}

\begin{Lem}\label{aperipheral}
Let $F\subset X_K$ be a properly embedded incompressible surface that has no $\partial$-parallel annulus component. Assume that $F$ has non-empty boundary components of zero slope.  If there exists a page $\Sigma_\theta$ such that $F\cap \Sigma_\theta$ consists of peripheral curves in $\Sigma_{\theta}$, then $F$ is isotopic to a union of pages.
\end{Lem}
\begin{proof}
Isotope $F$ to intersect $\Sigma_{\theta}$ minimally. By the previous lemma, it suffices to show that $F$ is disjoint from $\Sigma_{\theta}$. Assume for a contradiction that $F$ is not disjoint from $\Sigma_{\theta}$. Let us define $N=X_K\setminus \mathring{N}(\Sigma_{\theta})$ and $S=F\cap N=F\setminus \mathring{N}(\Sigma_{\theta})$.\vspace{1mm}

\noindent \textbf{Claim.} $S$ is incompressible in $N$.\vspace{1mm}

\noindent\textit{Proof.} Assume for a contradiction that $S$ is compressible in $N$. Choose a compressing disk $D$ for $S$ and let $\gamma=\partial D=D\cap S$. Since $F$ is an incompressible surface, $\gamma$ bounds a disk $E\subset F$ which does not lie in $S$. Therefore, $E$ intersects $\Sigma_\theta$, and a component $\delta$ of $E\cap \Sigma_{\theta}\subset F\cap \Sigma_{\theta}$ is peripheral in $\Sigma_{\theta}$ by assumption. Since the peripheral curve $\delta\subset \Sigma_{\theta}$ bounds a disk in $E$, we deduce that $\partial \Sigma_{\theta}$ bounds a disk in $X_K$. This implies that $K$ is the unknot in $M=S^3$, which contradicts the assumption that $M$ has a strongly irreducible Heegaard splitting. \qed\vspace{1mm}

By assumption, $\partial S$ is peripheral in the horizontal boundary $\Sigma\times \{0,1\}$ of $N\cong \Sigma\times I$. 
By Lemma \ref{standa}, we deduce that each component $S$ is either a page or a $\partial$-parallel annulus in $N$. It follows that the intersection of $F$ with the page $\Sigma_\theta$ consists of peripheral curves in $F$. Let $\gamma\subset F\cap \Sigma_{\theta}$ be an outermost curve of intersection, which cuts off an annulus $A$ from $F$. We can isotope $F$ in $X_K$ to eliminate $\gamma$ from the intersection $F\cap \Sigma_{\theta}$, which contardicts the minimality assumption.
\end{proof}

\begin{Lem}\label{zeroslope}
Let $F\subset X_K$ be a properly embedded incompressible surface with non-empty boundary components of zero slope. If $F$ is not isotopic to a union of pages and boundary-parallel annuli, then $d_{\C}(\phi)\le -\chi (F)$.
\end{Lem}
\begin{proof}
Fix a connected component $F'$ of $F$ which is neither a $\partial$-parallel annulus nor isotopic to a page. Since $F$ is incompressible, $F'$ is also incompressible. By Lemma \ref{standa}, $F'$ intersects every page $\Sigma_\theta$ of $K$. By Theorem 4 in \cite{Thurs}, we can isotope $F'$ in $X_K$ so that, away from its boundary, $F'$ is transverse to all but $m=-\chi(F')\le -\chi(F)$ pages, say $\Sigma_{\theta_1}, \ldots, \Sigma_{\theta_m}$, where $0<\theta_1< \ldots < \theta_m<2\pi$, and $F'$ is transverse to each $\Sigma_{\theta_i}$ except for a single saddle tangency.
Choose numbers  $t_0=0<t_1<\ldots<t_{m-1}<t_m = 2\pi$ such that each $t_i$ is in $(\theta_i,\theta_{i+1})$ for $i=1,\ldots, m-1$. It follows that 
\begin{enumerate}
	\item Each simple closed curve in $F'\cap \Sigma_{t_i}$ is non-trivial in $\Sigma_\theta$, since a trivial curve would yield a center tangency, and
	\item For any $i=0,\ldots,m$, at least one curve of intersection in $F'\cap \Sigma_{t_i}$ is non-peripheral in $\Sigma_{\theta}$, for otherwise $F'$ would be a $\partial$-parallel annulus or isotopic to a page by Lemma \ref{aperipheral}.
\end{enumerate}
For $i=0,1,\ldots,m$, choose a curve $\alpha_i$ in each $F'\cap \Sigma_{t_i}$ that is essential in $\Sigma_{t_i}$ while ensuring that $\phi(\alpha_0)=\alpha_m$. Observe that for each $i=0,\ldots,m-1$, there is only a single saddle tangency of $F'$ in $\Sigma\times [t_i,t_{i+1}]$. Therefore, we can isotope $\alpha_{i+1}$ into $\Sigma_{t_{i}}$ so that it is disjoint from $\alpha_i$. In other words, for each $i=0,\ldots,m-1$, we have $d_\C(\alpha_i,\alpha_{i+1})\le 1$. Thus, by the triangle inequality, we get
$$d_{\C}(\phi)\le d_{\C}(\alpha_0,\phi(\alpha_0))= d_{\C}(\alpha_0,\alpha_m)\le \sum^{m-1}_{i=0}d_\C(\alpha_{i},\alpha_{i+1})\le m=-\chi(F')\le -\chi(F),$$
as desired.
\end{proof}
Before the proof of Theorem \ref{KinP} for non-primitive knots, we will introduce one more lemma.
\begin{Lem}\label{core}
Let $(P,U,V)$ be a strongly irreducible Heegaard splitting in $M$ and $K$ a knot that lies in $P$. If the surface $F = P\setminus \mathring{N}(K)$ can be compressed in $U$ or $V$ to an annulus $A$ that is $\partial$-parallel in $X_K$, then $K$ is primitive in $U$ or $V$, respectively.
\end{Lem}
\begin{proof}
The proof is symmetric with respect to $U$ and $V$. Therefore, we will give a proof only for $U$. Since $A$ is $\partial$-parallel in $X_K$, it is $\partial$-compressible. Let $\Delta\subset X_K$ be a $\partial$-compressing disk for $A$, where $\partial \Delta$ is a union of arcs $\alpha$ and $\beta$ such that $\alpha=\partial \Delta\cap A$ and $\beta=\partial \Delta\cap \partial X_K$. Isotope $\alpha$ away from the disks in $A$ that are introduced by the compressions of $F$ so that $\alpha$ lies in $P$ (when we undo the compressions).

Let $B$ be the annulus component of $\partial X_K\setminus \partial A$ that contains $\beta$. Now the annulus $B$ equals either $\partial X_K\cap U$ or $\partial X_K\cap V$. In each case, we can find a half disk $\Delta'$ in $N(K)\cap U$ (or in $N(K)\cap V$) such that $\partial \Delta'=\beta\cup \beta'$, where $\beta'$ is a spanning arc for the annulus $B'=N(K)\cap P$. (This is because the core of $B$ is an intergral slope in $\partial X_K$.) Concatenating $\Delta$ and $\Delta'$ along $\beta$, we obtain a disk $D\subset M$ such that $\partial D=\alpha\cup \beta'$ is a simple closed curve in $P$ that intersects $K$ (which is the core of $B'$) exactly once. Since $P$ is strongly irreducible, it follows from Lemma \ref{nonest} that $\alpha\cup \beta'$ bounds a disk $D'\subset U$. Finally, since $\partial D'$ intersects $K$ exactly once, $K$ is primitive in $U$.
\end{proof}
We conclude this section with the proof of Theorem \ref{KinP} when $K$ is not primitive.

\subsection{Proof of Theorem \ref{KinP} for non-primitive knots.}
Assume that the assumptions stated in Theorem \ref{KinP} hold and $K$ is primitive in neither $U$ nor $V$. Consider the properly embedded surface $F=P\setminus \mathring N(K)$ in $X_K$. Notice that $F$ has two boundary components of a non-meridional slope in $\partial X_K$. Since $F$ is obtained by removing an annulus from $P$, we have $\chi(F)=\chi(P)=2-2g$. Now we provide the proof by a case analysis depending on the compressibility of $F$ in $X_K$, and in each case we show that either Condition (1) or Condition (2) asserted in Theorem \ref{KinP} holds. \v

\noindent \textbf{Case 1.} $F$ is incompressible in $X_K$: In this case, we have two subcases depending on the boundary slopes of $F$.\vspace{2mm}

\noindent \textbf{Subcase 1.} $F$ realizes the zero slope: In this case, first note that $F$ cannot be an annulus beacuse that would imply $P$ is a torus, which is ruled out by the assumption that $g(P)\ge 2$. Then, by Lemma \ref{zeroslope}, either $d_\C(\phi)\le-\chi(F)=2g-2$, i.e., (2) holds, or $F$ is isotopic to a union of two pages. If $F$ is isotopic to a union of two pages, then $P$ is isotopic to the Heegaard surface induced by $K$, i.e., (1) holds.\vspace{2mm}

\noindent \textbf{Subcase 2.} $F$ realizes a non-zero slope: In this case, it directly follows from Lemma \ref{nonzeroslope} that $d_\A(\phi)\le -\chi(F)=2g-2$, and hence (2) holds.
\vspace{2mm}

\noindent\textbf{Case 2.} $F$ is compressible in $X_K$: In this case, by Lemma \ref{nonest}, there exists a compressing disk of $F$ that lies in one of the handlebodies $U$ or $V$. With no loss of generality, assume that there exists a compressing disk for $F$ in $U$. Let $G\subset X_K$ be the surface obtained by maximally compressing $F$ in $U$.
\begin{Clm}\label{G}
Every non-sphere component of $G$ is incompressible in $X_K$.
\end{Clm}
\begin{proof} Assume for a contradiction that there is a non-sphere component $S$ of $G$ that is compressible in $X_K$. Let $\gamma\subset S$ be a curve that bounds a compressing disk $D$ for $S$ in $X_K$. We can isotope $\gamma$ into $F\cap S$ because $S\setminus F$ is a union of disks in $S$ (which are introduced by the compressions of $F$ in $U$). Hence, $\gamma$ is an essential curve on the strongly irreducible Heegaard splitting $P$ that bounds a disk $D$ in $X_K$. By Lemma \ref{nonest}, $D$ can be assumed to lie in either $U$ or $V$ completely. If $D\subset U$, then $F$ is not maximally compressed in $U$, which is a contradiction. If $D\subset V$, then $D$ is an essential disk in $V$ that is disjoint from the compressing disks of $F$ in $U$, which contradicts the strong irreducibility of $P$.
\end{proof}
Notice that $\chi(G)>\chi(F)=2-2g$ and $G$ has two boundary components since it is obtained from $F$ by compressions. Let $S$ be the union of the components of $G$ that contains the boundary. Since $G$ is incompressible in $X_K$, so is $S$. Notice that $S$ cannot be a union of two pages, for otherwise the strongly irreducible Heegaard surface $P$ would be compressed in $U$ into the Heegaard surface induced by $K$, which is impossible by Lemma \ref{compressions}. Moreover, $S$ is not a $\partial$-parallel annulus in $X_K$, for otherwise Lemma \ref{core} would imply that $K$ is primitive in $U$.
Finally, depending on the boundary slope of $S$, Lemmas \ref{nonzeroslope} and \ref{zeroslope} imply that we have $d_{\AC}(\phi)\le -\chi (S)\le -\chi (G)< 2g-2$, i.e., (2) holds.\qed


\section{Primitive Fibered Knots on Strongly Irreducible Heegaard Surfaces} \label{Primitive}

The discussion so far leaves one case to discuss for a complete proof of Theorem \ref{KinP}: $(P,U,V)$ is a strongly irreducible Heegaard splitting and $K$ is a fibered knot in $M$ such that $K$ is a core in $U$ or $V$ up to isotopy. In this section, we will prove Theorem \ref{KinP} under these assumptions.

In Section \ref{EssentialSurfaces}, we showed that if there exists a closed surface $S\subset X_K$ that is incompressible in $M$, then $d_{\C}(\phi)\le 2g(S)-2$. In this section, we will achieve a similar complexity bound when there is a Heegaard splitting $P\subset X_K$ that is strongly irreducible in $M$. We will generalize the result of Section \ref{EssentialSurfaces} from closed incompressible surfaces to strongly irreducible Heegaard splittings, by using the double sweepout technique along with a labelling, similar to \cite{BachmanSchleimer} and \cite{Li}. Some arguments will be very similar to Section \ref{ThinPosition} and we will give short explanations for such arguments.  We will also refer to the figures of Section \ref{ThinPosition}. First, we will introduce literature, notation, and some useful lemmas. The proof will be presented at the end of this section.\vspace{2mm}

\subsection{Intersection graphics of surface families}

Assume that $(P,U,V)$ is a strongly irreducible Heegaard splitting of $M$ and $K$ is a core in $U$. We denote the Heegaard splitting of $X_K$ determined by $P$ by $(P,U',V)$, where  $U'$ is the compression-body obtained by removing an open tubular neighborhood $\mathring{N}(K)$ of $K$ from $U$. A \emph{spine} of $U'$, denoted by $G_{U'}$, is a wedge of $\partial_{-}U'=\partial X_K$ with a spine of a genus $g-1$ handlebody embedded in $U'$ such that $U'\setminus G_{U'}$ is homeomorphic to $P\times (0,1]$.

 For fixed spines $G_{U'}$ of $U'$ and $G_V$ of $V$, a \emph{sweepout} of the Heegaard splitting $(P,U',V)$ is a smooth function $H: P\times I\to X_K$ such that $H(P\times \{0\})= G_{U'}$, $H(P\times \{1\})= G_V$, and $H(P\times \{t\})$ is isotopic to $P$ for any $t\neq 0,1$. For simplicity, we will denote $H(P\times \{t\})$ by $P_t$, and we will not distinguish $H(P\times(0,1))$ from $P\times (0,1)$. For any $t\in (0,1)$, let
\begin{itemize}
	\item[(a)] $U_t'$ denote the compression-body $P\times [0,t]$ bounded by $P_t$ in $X_K$,
	\item[(b)] $U_t$ denote the handlebody $U_t'\cup N(K)$ bounded by $P_t$ in $M$, and
	\item[(c)] $V_t$ denote the handlebody $P\times [t,1]$ bounded by $P_t$ in $M$.
\end{itemize}
One can isotope the pages $\Sigma_\theta$ in $X_K$ so that they are \emph{standard} with respect to $P_t$, i.e., 
\begin{itemize}
	\item there exists a page $\Sigma_{\theta}$ that is transverse to the spine $P_1=G_{V}$, and 
	\item the pages are transverse to $\partial X_K$ and the level surfaces $P_t$ near the spine $P_0=G_{U'}$.
\end{itemize}
 Moreover, by Cerf theory \cite{Cerf}, the pages $\Sigma_\theta$ can be further isotoped so that the families $\Sigma_{\theta}$ and $P_t$ are in \emph{Cerf position}, that is, the set
$$\Lambda=\{(\theta,t)\in S^1\times (a,b)\,|\,\Sigma_\theta\text{ is not transverse to }P_t\}$$ 
is a one-dimensional graph in the open annulus $A=S^1\times (0,1)$ satisfying the properties (1)-(6) provided in Subsection \ref{Cerf}. Similar to Section \ref{ThinPosition}, the graph $\Lambda$ is called an \emph{intersection graphic} of the families $\Sigma_\theta$ and $P_t$, and a connected component of the $A\setminus \Lambda$ is called a \emph{region} of $A\setminus \Lambda$.\vspace{2mm}

\noindent\textbf{Assumption.} For the rest of this section, assume that $K$ is a fibered knot in $M$ with pages $\Sigma_\theta$ of genus greater than one, $(P,U,V)$ is a strongly irreducible Heegaard splitting of $M$, and $K\subset P$ is a core in $U$. Let $P_t$, $t\in[0,1]$, be a sweepout of $P$ in $X_K$ 
such that the families $\Sigma_{\theta}$ and $P_t$ of are in a Cerf position providing an intersection graphic $\Lambda$ in the annulus $A=S^1\times (0,1)$, satisfying the properties mentioned above.

\subsection{Labelling} 
We label a level surface $P_t$ with $U$ (resp. with $V$) if there exists a page $\Sigma_{\theta}$, which is transverse to $P_t$, such that every component of $\Sigma_{\theta}\cap P_t$ is an inessential curve in $\Sigma_{\theta}$ that is not disk-busting in the handlebody $U_t$ (resp. in $V_t$). Additionally, we label a region $R$ of $(A\setminus \Lambda)$ with $U$ (resp. with $V$) if there exists a point $(\theta,t)\in R$ such that every component of $\Sigma_{\theta}\cap P_t$ is an inessential curve in  $\Sigma_{\theta}$ that is not disk-busting in the handlebody $U_t$ (resp. in $V_t$).

In the proof of Theorem \ref{KinP} for a primitive knot $K$, we will eventually show that if $P$ is not isotopic in $M$ to the Heegaard surface induced by $K$, then there exists a level surface $P_s$, which is not labelled. Such a surface will behave similarly to a perfect surface in $X_K$ and help us achieve a complexity bound, similar to the proof of Theorem \ref{ThinTheorem}. In this subsection, we will prove the following lemma which serves that purpose.

\begin{Lem}\label{bothUV}
If there exists a level surface $P_t$ that is labelled with both $U$ and $V$, then $P_t$, and therefore $P$, is isotopic to the Heegaard surface induced by $K$.
\end{Lem}
\noindent Before proving the lemma, we will introduce a few claims that will be useful. Since we have already fixed a Heegaard splitting $(P,U,V)$ for $M$, we will denote the Heegaard splittings by $(H,X,Y)$  instead of $(P,U,V)$ in the statements, to avoid confusion.
\begin{Clm}\label{pagesandP}
If $(H,X,Y)$ is a Heegaard splitting of $M$, and $K\subset M$ is a core in $X$ (or $Y$), then any page of $K$ intersects $H$.
\end{Clm}
\begin{proof}
Assume for a contradiction that there exists a page $\Sigma_\theta$ of $K$ such that $\Sigma_\theta\cap H=\emptyset$. Since $K$ is a core in, say $X$, it follows that $H$ is a Heegaard splitting of $X_K$ that bounds the compression-body $X'=X\setminus \mathring{N}(K)$ on one side, and the handlebody $V$ on the other side.

By Dehn filling $X_K$ along the boundary of a page, we obtain a fibered three-manifold $\wh M$. Moreover, $H$ persists in $\wh M$ as a Heegaard surface since it bounds the pair of handlebodies $(\wh{X}, Y)$, where $\wh{X}=X'\cup\text{(the filling torus)}$. 
Hence, $H\subset \wh{M}$ is a Heegaard surface  that is disjoint from $\widehat\Sigma_\theta=\Sigma_\theta\cup\text{(a filling disk)}$. This implies that the fiber $\widehat \Sigma_\theta$ lies in a handlebody bounded by $H$ in $\wh M$, which contradicts the incompressibility of the fiber.
\end{proof}
\begin{Clm}\label{bicompressible} Assume that $(H,X,Y)$ is a strongly irreducible Heegaard splitting of $M$, and $K$ is a core in $X$ (or $Y$). Let $\Sigma_{\theta}$ be a page of $K$ such that $ \Sigma_\theta\cap H$ is a collection of simple closed curves that are peripheral in $\Sigma_{\theta}$. Then at least one component of $\Sigma_\theta\cap H$ is disk-busting in either $X$ or $Y$.
\end{Clm}
\begin{proof}
Assume for a contradiction that no curve in  $\Sigma_\theta\cap H$ is disk-busting in $X$ or $Y$. First note that $\Sigma_\theta\cap H$ is non-empty by Claim \ref{pagesandP}. Since all curves in $\Sigma_\theta\cap H$ are peripheral in $\Sigma_{\theta}$, there exists a component $K'$ of $\Sigma_\theta\cap H$ that cuts off an annulus $A$ from $\Sigma_\theta$ that contains all other curves of intersection. Observe that $K'\subset H$ is a fibered knot in $M$ as it is isotopic to $K$ (in $M$, but not necessarily in $U$). Moreover, $\Sigma_\theta' = \overline{\Sigma_\theta \setminus A}$ is a page of $K'$ that completely lies in one of the handlebodies, say $X$ without loss of generality. By assumption, $K'$ is not disk-busting in $X$. It follows that $F=H\setminus \mathring{N}(K')$ is a surface in the exterior of $K'$ that is disjoint from a page $\Sigma_{\theta}'$ and compressible in the handlebody $X$.

Let $G\subset X$ be the surface obtained by maximally compressing $F$ in $X$. By Claim \ref{G}, $G$ is incompressible in $X_{K'}$. Since $\Sigma_{\theta}'$ is incompressible, before compressing $F$ in $X$, we can isotope $\Sigma_{\theta}'$ away from the compressing disks that yield $G$. Therefore, we can assume that $G$ and $\Sigma_{\theta}'$ are disjoint. Now, let $S$ be the union of the components of $G$ which contain $\partial G=\partial F$, so $S\subset X_{K'}$ is an incompressible surface disjoint from the page $\Sigma_{\theta}'$ with two boundary components of the zero slope in $\partial X_{K'}$. By Lemma \ref{standa}, we have the following two possibilities for $S$, and both yield a contradiction.
\begin{enumerate}
	\item $S$ is a $\partial$-parallel annulus in $X_{K'}$: In this case, the fibered $K'\subset H$ is primitive, and hence it is a core, in $X$ by Lemma \ref{core}. Moreover, $\Sigma_{\theta}'$ is a page of $K'$ that is disjoint from $H$, which is impossible by Claim \ref{pagesandP}.
	\item $S$ is isotopic to a union of two pages: In this case, the union of $S$ with the annulus $B=H\cap N(K')$ yields a Heegaard surface $H'$ induced by $K'$. In other words, the strongly irreducible Heegaard surface $H\subset M$ can be compressed in $X$ to the Heegaard surface $H'$, which is impossible by Lemma \ref{compressions}.
\end{enumerate}
This completes the proof.
\end{proof}

\begin{Clm}\label{eliminate} Assume that $(H,X,Y)$ is a strongly irreducible Heegaard splitting of $M$, and $K$ is a core in $X$ (or $Y$). Let $\Sigma_{\theta}$ be a page of $K$ such that no component of $\,\Sigma_\theta\cap H$ bounds an essential disk in $X$ or $Y$.  Then we can isotope $H$ so that every component of $\,\Sigma_{\theta}\cap H$ is non-trivial in both $\Sigma_{\theta}$ and $H$.
\end{Clm}
\begin{proof}
First note that any curve $\gamma\subset\Sigma_\theta\cap H$ that is trivial in $\Sigma_\theta$ is also trivial in $H$. Otherwise, by Lemma \ref{nonest}, $\gamma$ bounds an essential disk $D$ in $X$ or $Y$, which contradicts the assumption. On the other hand, any curve $\gamma\subset\Sigma_\theta\cap H$ that is trivial in $H$ is also trivial in $\Sigma_\theta$ by Lemma \ref{trivial} (basically because $\Sigma_\theta$ is incompressible). Therefore, we can isotope $H$ to eliminate trivial curves from the intersection by applying the standard ``innermost intersection curve'' argument.
\end{proof}

Now we are ready to prove the main lemma of this subsection.
\v

\noindent \textit{Proof of Lemma~\ref{bothUV}.} Let $P_t$ be labelled with both $U$ and $V$, i.e., there exist pages $\Sigma_U$ and $\Sigma_V$ such that 
\begin{itemize}
	\item[(a)] every curve in $\Sigma_U\cap P_t$ is inessential in $\Sigma_U$ and not disk-busting in $U_t$;
	\item[(b)] every curve in $\Sigma_V \cap P_t$ is inessential in  $\Sigma_V$ and not disk-busting in $V_t$. 
\end{itemize}\vspace{2mm}

\noindent\textbf{Claim.} Both $\Sigma_U\cap P_t$ and $\Sigma_V\cap P_t$ have no component that bounds an essential disk in $U_t$ or $V_t$.  \vspace{2mm}

\noindent\textit{Proof of the claim.} Assume for a contradiction that $\Sigma_U\cap P_t$ has a component $\gamma_U$ that bounds an essential disk in $U_t$ or $V_t$. By labelling, $\gamma_U$ is not disk-busting in $U_t$. Since $P_t$ is strongly irreducible, we deduce that $\gamma_U$ cannot bound a disk in $V_t$. So, $\gamma_U$ bounds an essential disk $D_U\subset U_t$. Now we have two cases depending on $\Sigma_V\cap P_t$, and both yield a contradiction.\vspace{2mm}

\noindent \textbf{Case 1.} $\Sigma_V\cap P_t$ has a component $\gamma_V$ that bounds an essential disk in $U_t$ or $V_t$: In this case, since $\gamma_V$ is not disk-busting in $V_t$ and $P_t$ is strongly irreducible, we deduce that the curve $\gamma_V$ bounds an essential disk $D_V$ in $V_t$. Since $\Sigma_{U}$ and $\Sigma_{V}$ are pages (possibly equal pages with no self-intersection) that do not intersect each other, we deduce that the essential disks $D_U\subset U_t$ and $D_V\subset V_t$ do not intersect, which contradicts the strong irreducibility of $P_t$.\vspace{2mm}

\noindent \textbf{Case 2.} $\Sigma_V\cap P_t$ has no component that bounds an essential disk in $U_t$ or $V_t$: In this case, by Claim \ref{eliminate}, we can isotope $P_t$ so that $\Sigma_V\cap P_t$ contains no trivial curves. By the labelling, $\Sigma_V \cap  P_t$ contains no essential curves in $\Sigma_V$, and so, after the isotopy, $\Sigma_V\cap P_t$ is a collection of peripheral curves that are not disk-busting in $V_t$. On the other hand, since $\gamma_U\subset \Sigma_U\cap P_t$ bounds an essential disk that is disjoint from $\Sigma_V\cap P_t$, we deduce that $\Sigma_V\cap P_t$ is not disk-busting in $U_t$ either. However, this is impossible by Claim \ref{bicompressible}.\qed\vspace{2mm}

It follows from Claim \ref{eliminate} that we can isotope $P_t$ to eliminate all simple closed curves of $\Sigma_U\cap P_t$ and $\Sigma_V\cap P_t$ that are trivial in $\Sigma_U$ and $\Sigma_V$, respectively. After the isotopy, $\Sigma_U\cap P_t$ (resp. $\Sigma_V\cap P_t$) is a collection of peripheral cuves in $\Sigma_U$ (resp. in $\Sigma_V$). Since $\Sigma_U\cap P_t$  is not disk-busting in $U_t$, by Claim \ref{bicompressible}, we deduce that it has a component $\gamma_V$ that is disk-busting in $V_t$. Similarly, $\Sigma_V\cap P_t$ has a component $\gamma_U$ that is disk-busting in $U_t$.

Now we will show that $P_t$ is isotopic to the Heegaard surface induced by $K$. Let $N$ be the complement of an open tubular neighborhood $\mathring{N}(\Sigma_{U}\cup \Sigma_{V})$ in $X_K$ and $F=P_t\cap N$. Notice that each component of $N$ is homeomorphic to $\Sigma\times I$.

First we prove that $F$ is incompressible in $N$. Assume for a contradiction that $F$ is compressible with a compressing disk $D$. Then $\alpha = \partial D$ can be regarded as an essential curve in $P_t$ that bounds a disk in $M$. By Lemma \ref{nonest}, $\alpha$ bounds an essential disk in $U_t$ or $V_t$, which is impossible because $\gamma_U$ and $\gamma_V$ are disk-busting in $U$ and $V$, respectively.

Note that each component of $\partial F$ is peripehral in the horizontal boundary components of $N$. Therefore, $F$ can be isotoped in $N$ so that $\partial F$ lies in the vertical boundary components $\partial\Sigma\times I$. It follows from Lemma \ref{standa} that $F=P_t\cap N$ is isotopic to a union of pages and $\partial$-parallel annuli in $N$. We deduce that $P_t$ is isotopic in $M$ to a union of a collection pages and annuli. Since $P_t$ is not a torus, it contains a subsurface that is homeomorphic to a page. Since the only closed connected surface that can be constructued as a union of pages and annuli is the Heegaard surface induced by $K$, it follows that $P_t$ is isotopic to the Heegaard surface induced by $K$.
 \qed



\subsection{A special level} In the previous subsection, we showed that if there exists a level surface $P_t$ that receives both labels $U$ and $V$, then $P$ is induced by the fibered knot $K$, which is one of the possible conslusions in Thorem \ref{KinP}. In this subsection, we will show that if $P$ is not induced by $K$, then there exists a level $P_s$ that does not receive a label and this surface will provide the complexity bound stated in Theorem \ref{KinP}.

\begin{Clm}\label{standar}For $\delta>0$ sufficiently small, $P_{\delta}$ is labelled with $U$, and $P_{1-\delta}$ is labelled with $V$.
\end{Clm}
\begin{proof} This is basically because $\Sigma_\theta$ and $P_t$ have standard intersection near the spines. 

For $t$ values near $0$, every curve $\gamma \subset\Sigma_\theta\cap P_t$ is inessential in $\Sigma_\theta$. If $\gamma$ is trivial in $\Sigma_\theta$, then it bounds a disk in $U_t$. If $\gamma$ is peripheral in $\Sigma_\theta$, then it is primitive in $U_t$. In both cases, $\gamma$ is not disk-busting. So, $P_t$ is labelled with $U$.

For $t$ values near $1$, every curve $\gamma \subset\Sigma_\theta\cap P_t$ is inessential in $\Sigma_\theta$ and bounds a disk in $V_t$. So, $P_t$ is labelled with $V$.
\end{proof}

\begin{Lem}
If $P$ is not induced by $K$, there exists a level surface $P_s$ that is not labelled.
\end{Lem}
\begin{proof}
Let $s:=\sup\{t\in(0,1)| P_t \text{ is labelled with $U$}\}$. The lemma follows from the following observations, which are similar to some arguments we introduced in Section \ref{ThinPosition}:
\begin{enumerate}
\item $0<s$: This is because $P_\delta$ is labelled with $U$ for $\delta>0$ sufficiently small.
\item $s<1$: If $s=1$, then there are $t$ values arbitrarily close to $1$ such that $P_t$ receives both labels. Hence, by Lemma \ref{bothUV}, $P$ is induced by $K$ up to isotopy.
\item $P_s$ is not labelled with $U$: If $P_s$ is labelled with $U$, then for small $\delta>0$, $P_{s+\delta}$ is labelled $U$, which contradicts the definition of $s$.
\item For any $\epsilon>0$, there exists a $t\in (s-\epsilon,s)$ such that $P_t$ is labelled with $U$: If this does not hold, $s$ cannot be the supremum of the parameters of $U$-labelled levels.
\item $P_s$ is not labelled with $V$: If $P_s$ is labelled with $V$, then for small $\delta>0$, $P_{s-\delta}$ is labelled $V$. Hence, there are $t$ values arbitrarily close to $s$ which are labelled with both $U$ and $V$. Therefore, by Lemma \ref{bothUV}, $P$ is induced by $K$ up to isotopy.
\end{enumerate}
Thus, $P_s$ is an unlabelled level as stated in observations (3) and (5).
\end{proof}
Now let us fix a level surface $P_s$ that is unlabelled. Unlike Section \ref{ThinPosition}, we do not necessarily specify $s$ to be $\sup\{t\in(0,1)\,|\, P_t \text{ is labelled with $U$}\}$ .

\begin{Lem}\label{finally}
If $P$ is not induced by $K$, then for any angle $\theta$ such that $\Sigma_\theta\cap P_s$ is transversal, there exists a component of $\Sigma_\theta\cap P_s$ that is essential in $\Sigma_{\theta}$.
\end{Lem}
\begin{proof} We prove the contrapositive. Assume that there exists a page $\Sigma_{\theta}$ such that $\Sigma_{\theta}\cap P_s$ is transversal and inessential in $\Sigma_{\theta}$. Since $P_s$ is not labelled, it follows that there exist components $\gamma_U$ and $\gamma_V$ in $\Sigma_{\theta}\cap P_s$ such that $\gamma_U$ is disk-busting in $U_s$ and $\gamma_V$ is disk busting in $V_s$. This implies that no component of $\Sigma_{\theta}\cap P_s$ bounds an essential disk in $U_s$ or $V_s$. Thus, by Lemma \ref{eliminate}, we can isotope $P_s$ to eliminate trivial curves of intersection so that $\Sigma_\theta\cap P_s$ consists of curves that are peripheral in $\Sigma_{\theta}$. After the isotopy, $\Sigma_\theta\cap P_s$ still contains curves $\gamma_U$ and $\gamma_V$ that are disk-busting in $U_s$ and $V_s$, respectively. Similar to the argument at the end of the proof of Lemma \ref{bothUV}, this implies that $P_s$, and therefore $P$, is induced by $K$. We shortly explain it.
	
Let $N\cong \Sigma\times I$ be the complement of an open tubular neighborhood $\mathring{N}(\Sigma_{\theta})$ in $X_K$ and $F=P_s\cap N$. Similarly, $F$ is incompressible in $N$ and it is isotopic to a union of pages and $\partial$-parallel annuli in $N$ (by Lemma \ref{standa}). We deduce that $P_s$ is isotopic in $M$ to a union of a collection pages and annuli. Since $P_s$ is not a torus, it contains a subsurface that is homeomorphic to a page. Since the only closed connected surface that can be constructued as a union of pages and annuli is the Heegaard surface induced by $K$, it follows that $P_s$ is isotopic to the Heegaard surface induced by $K$.
\end{proof}
The discussion so far points out that if Conclusion (1) of Theorem \ref{KinP} does not hold, then there exists a level surface $P_s$ that is unlabelled and the intersection of this surface with any transverse page $\Sigma_{\theta}$ contains a simple closed curve that is essential in that page. Before the proof of Theorem \ref{KinP} for primitive knots, we will state and prove two more lemmas, which will be helpful to prove that such a surface imposes a complexity bound.
\begin{Lem}\label{bo}
Assume that $P$ is not induced by $K$. If $P_s$ is an unlabelled level surface such that there is no vertex of $\Lambda$ on the horizontal circle $C_s$, then $d_\C(\phi)\le 2g-2$.
\end{Lem}
\begin{proof}
If $C_s$ contains no vertex, then $P_s$ is a regular surface. Moreover, since $P$ is not induced by $K$, Lemma \ref{finally} implies that for any angle $\theta$, if $\Sigma_{\theta}\cap P_s$ is transversal, then it contains a curve that is essential in $\Sigma_{\theta}$. In other words, $P_s$ is a perfect surface in $X_K$ (see Definition \ref{perfect}). Hence, by Lemma \ref{closedbound}, $d_\C(\phi)\le -\chi(P_s)=2g-2$.
\end{proof}

\begin{Lem}\label{cr}
Assume that $P$ is not induced by $K$. If $P_s$ is an unlabelled level surface such that there is a birth-and-death vertex of $\Lambda$ on $C_s$, then $d_\C(\phi)\le 2g-2$.
\end{Lem}
\begin{proof}
If there exists a birth-and-death vertex of $\Lambda$ on $C_s$, then
we find a sufficiently small $\epsilon>0$ such that $P_{s-\epsilon}$ (or $P_{s+\epsilon}$) is unlabelled, and $C_{s-\epsilon}$ contains no vertex. 
Hence, $P_{s-\epsilon}$ satisfies the hypothesis of the previous lemma, and the complexity bound follows.
\end{proof}

\subsection{Proof of Theorem~\ref{KinP} for primitive knots.} Assume that $(P,U,V)$ and $K$ are as in the statement of Theorem \ref{KinP} and $K$ is a core in either $U$ or $V$ up to isotopy. We will assume that Condition (1) does not hold and show that Condition (2) holds. So, assume $P$ is not isotopic to the Heegaard surface induced by $K$. It follows from Lemma \ref{finally} that there is an unlabelled level $P_s$ such that for any page $\Sigma_\theta$ that is transversal to $P_s$, there exists a curve $\alpha\subset \Sigma_\theta\cap P_s$ that is essential in $\Sigma_{\theta}$. Moreover, by Lemmas \ref{bo} and \ref{cr}, we can assume that the horizontal circle $C_s$ contains a crossing vertex $(\psi,s)$ for otherwise we obtain $d_\C(\phi)\le 2g-2$, i.e., (2) holds. Under these assumptions, the following facts follow from the arguments of Section \ref{ThinPosition}:
\begin{enumerate}
	\item $\Sigma_{\psi}$ and $P_s$ intersect transversely except for two entangled saddle tangencies.
	\item If $\wh F$ is the component of $P_s\cap \Sigma\times [\psi-\epsilon,\psi+\epsilon]$ that contains the saddle tangencies (for $\epsilon>0$ small), then every component of $\Sigma_{\psi\pm\epsilon} \cap \wh F$ is non-trivial in $\Sigma_{\psi\pm\epsilon}$.
	\item $P_s$ has at most $m=-\chi(P_s)-2$ essential saddle tangencies to distinct pages in $\Sigma\times ([0,\psi-\epsilon]\cup [\psi+\epsilon, 2\pi])$.
\end{enumerate}
By rotating the pages of $K$ and reparametrizing $\theta$, if necessary, we can assume that $\Sigma_0$ and $P_s$ intersect transversely, and $\Sigma\times (\psi,2\pi)$ contains no tangencies of $P_s$.

Now let $S$ be the preimage of $P_s$ under the quotient map $q: \Sigma\times [0,2\pi]\to X_K$, which maps $\Sigma\times\{\theta\}$ to $\Sigma_\theta$ in the natural way. For simplicity, we will not distinguish $\Sigma\times\{\theta\}$ and $\Sigma_\theta$. Let $t_0=0<t_1<\ldots<t_m<\psi<t_{m+1}=2\pi$ be angles such that, for $i=1,\ldots,m-1$, $\Sigma_{t_i}$ and $P_s$ are transverse, and $\Sigma\times [t_i,t_{i+1}]$ contains a single essential saddle tangency of $P_s$. Furthermore, for $i=0,1,\ldots,m+1$, fix simple closed curves $\alpha_i\subset \Sigma_{t_i}\cap P_s$ that are essential in $\Sigma_{t_i}$, while ensuring $\phi(\alpha_0)=\alpha_{m+1}$. The following claims will complete the proof. Recall that $m=-\chi(P_s)-2$ in the statements. \vspace{2mm}

\noindent\textbf{Claim.} For $i=0,\ldots,m-1$, we have $d_\C(\alpha_i,\alpha_{i+1})\le 1$. \vspace{1mm}

\noindent\textit{Proof.} This is basically because there exists a single essential saddle tangency of $P_s$ in $\Sigma\times[t_i,t_{i+1}]$. If one of the curves, say $\alpha_i$, does not interact with the saddle tangency, then $F\cap \Sigma_{t_{i+1}}$ contains a curve that is isotopic to $\alpha_i$. Therefore, either $\alpha_i=\alpha_{i+1}$ or they are disjoint, and we get $d_\C(\alpha_i,\alpha_{i+1})\le 1$. So, we can assume that both curves interact with the essential saddle tangency of $P_s$ in $\Sigma\times[t_i,t_{i+1}]$. In this case, the saddle tangency guides an isotopy of $\alpha_i$ into $\Sigma_{t_{i+1}}$ such that $\alpha_i$ and $\alpha_{i+1}$ are disjoint, and we get $d_\C(\alpha_i,\alpha_{i+1})\le 1$.\qed\vspace{2mm}
	
\noindent\textbf{Claim.}  We have $d_\C(\alpha_{m},\alpha_{m+1})\le 2$. \vspace{1mm}

\noindent\textit{Proof.} Similar to the previous claim, we can assume that $\alpha_{m}$ and $\alpha_{m+1}$ interact with the entangled saddle of $P_s$ to $\Sigma_{\psi}$. Otherwise, we similarly get $d_\C(\alpha_{m},\alpha_{m+1})\le 1$. Let $G$ be the singular component of $\Sigma_\psi\cap P_s$, so $G$ is a graph embedded in $\Sigma_\psi$ with two vertices of valence 4. Since $g(\Sigma_{\psi})\ge 2$, we deduce that $G$ does not fill $\Sigma_\psi$. On the other hand, $\alpha_{m}$ and $\alpha_{m+1}$ have isotopic copies that lie in a neighborhood $N(G)\subset \Sigma_\psi$. Since $G$ does not fill $\Sigma_\psi$, there exists an essential curve $\beta$ outside $N(G)$. Therefore, $\beta$ is disjoint from the isotopic copies of $\alpha_{m}$ and $\alpha_{m+1}$ in $\Sigma_{\psi}$, and we obtain $d_\C(\alpha_{m},\alpha_{m+1})\le 2$.\qed\v

Finally, it follows from the last two claims that
$$d_{\C}(\phi)\le d_{\C}(\alpha_0,\phi(\alpha_0))= d_{\C}(\alpha_0,\alpha_{m+1})\le \sum^{m}_{i=0}d_\C(\alpha_{i},\alpha_{i+1})\le m+2=-\chi(P_s)=2g-2,$$
as desired. \qed

\end{document}